\documentclass[11pt]{article}
\usepackage[margin=1in]{geometry}
\usepackage{amsmath,amsthm,amsfonts,amssymb, amsbsy}
\usepackage{enumerate,color,xcolor}
\usepackage{graphicx}
\usepackage{subfigure}
\usepackage{url}

\numberwithin{equation}{section}
\theoremstyle{plain}

\newcommand*\Bell{\ensuremath{\boldsymbol\ell}}

\newtheorem{theorem}{Theorem}

\newtheorem{definition}[theorem]{Definition}

\newtheorem{lemma}[theorem]{Lemma}

\newtheorem{proposition}[theorem]{Proposition}

\newtheorem{remark}[theorem]{Remark}
\newtheorem{assumption}[theorem]{Assumption}
\newcommand{\gao}[1]{\textcolor{black}{#1}}

\usepackage{comment}

\begin{document}

\begin{center}
  \Large \bf  Affine Point Processes: Refinements to Large-Time Asymptotics
\end{center}

\author{}
\begin{center}
{Xuefeng
  Gao}\,\footnote{Department of Systems
    Engineering and Engineering Management, The Chinese University of Hong Kong, Shatin, N.T. Hong Kong;
    xfgao@se.cuhk.edu.hk},
  Lingjiong Zhu\,\footnote{Department of Mathematics, Florida State University, 1017 Academic Way, Tallahassee, FL-32306, United States of America; zhu@math.fsu.edu.
  }
\end{center}


\begin{abstract}%
Affine point processes are a class of simple point processes with self- and mutually-exciting properties, 
and they have found useful applications in several areas.
 In this paper,
we obtain large-time asymptotic expansions in large deviations and refined central limit theorem for affine point processes, using the framework of
mod-$\phi$ convergence. Our results extend the large-time limit theorems in [Zhang et al. 2015. \textit{Math. Oper. Res.}
\textbf{40}(4), 797-819]. The resulting explicit approximations for large deviation probabilities and tail expectations can be used as an alternative to importance sampling Monte Carlo
simulations. Numerical experiments illustrate our results. 
\end{abstract}




\section{Introduction} 

Affine point processes, as described in \cite{Errais}, are versatile Markovian models often
used to capture the ``clustering'' feature of event arrivals. A point process is affine if 
its intensity is an affine function of an affine jump diffusion (see \cite{Duffie2000}) and its jump size are drawn from a fixed distribution.
The Poisson process and Markovian Hawkes processes (see \cite{Hawkes}) are special cases of affine point processes.
The components of affine point processes are self-- and mutually--exciting, i.e., the occurrence of an earlier point affects the probability of occurrence of later points of all types. In addition, 
this family of point processes is tractable as the Fourier transform of an affine point process is
an exponentially affine function of the driving jump-diffusion, and the coefficients of these affine functions solve certain ordinary differential equations (ODEs); see, e.g., \cite{Errais}.
Hence, affine point processes and the special case Markovian Hawkes processes have become popular models in several applications. Examples include finance and economics (\cite{AitSahalia2015, Bowsher2007, Errais, Hewlett2006}), social networks (\cite{Fox2016, Pinto2015}), queueing (\cite{Daw2018, GaoZhu2016b, Koops2018}) and others.


In a recent work, Zhang et al. \cite{Zhang} analyzed the large--time asymptotics of affine point processes. Specifically,
they consider a multi-dimensional affine point process $(L_1, \ldots, L_n)$, and they
establish a central limit theorem and a large deviation principle for $V(t):=\sum_{i=1}^{n}L_{i}(t)$ as $t \rightarrow \infty$.
The large deviations result in \cite{Zhang} is
of \textit{logarithmic asymptotics} type, i.e.,
they obtained the limit $\mathcal{I}(R):=-\lim_{t\rightarrow\infty}\frac{1}{t}\log\mathbb{P}(V(t) \geq Rt)$
for $R$ being a suitably large number. 
Relying on the large deviations result, they developed an asymptotically optimal importance sampling algorithm to estimate small tail probabilities $\mathbb{P}(V(t) \geq Rt)$.

In this paper, we derive explicit asymptotic expansions of large deviation probabilities $\mathbb{P}(V(t) \geq Rt)$ and more generally, tail expectations $\mathbb{E}\left[g(V(t)- R t) \cdot 1_{V(t) \geq R t} \right]$  for a wide class of real-valued functions $g$ as $t \rightarrow \infty$. See Theorem~\ref{thm:preciseLDP}.  In particular, our results
improve the logarithmic-scale large deviations asymptotics in \cite{Zhang}. By truncating the asymptotic expansions, we also obtain explicit approximations for large deviation probabilities and tail expectations associated with affine point processes. 
These approximations can be useful when exact computations or numerical inversion methods have difficulties,
e.g., when the ODEs governing the transform of an affine point process do not have closed-form solutions, and the probability of the event $\{V(t) \geq R t\}$ is very small. 
Our numerical studies 
illustrate that these analytical approximations can be accurate in the large-time regime, and they are faster to evaluate than importance sampling simulations in \cite{Zhang}. So these approximations can serve as alternatives or complementary tools to the 
Monte Carlo simulation which can be computer resource extensive.

Deriving asymptotic expansions of large deviation probabilities and tail expectations for affine point processes in the large--time regime is a difficult problem, because such point processes have complex self-- and mutually--exciting dependence structures,
and also we aim to obtain asymptotic expansions beyond the logarithmic asymptotics. In the literature,
asymptotic expansions for large deviation probabilities date back to \cite{Bahadur1960, Cramer}, where they derived expansions
for tail probabilities of sums of independent and identically distributed (i.i.d) random variables, which can be applied to Poisson processes. See also \cite{Dembo, Petrov}.
These asymptotic expansions go beyond the logarithmic asymptotics known as the large
deviation principle in \cite{VaradhanLDP}, also known as the Donsker--Varadhan type large deviations (see e.g. \cite{Donsker}).
Beyond the i.i.d case, there are some very general and relatively easy-to-check conditions
to guarantee the large deviation principle,
e.g., G\"{a}rtner-Ellis theorem, which is used in \cite{Zhang} to obtain the large deviation principle for affine point processes.
Nevertheless, there are not many general results for asymptotic expansions beyond the logarithmic asymptotics.
Chaganty and Sethuraman \cite{CS93} extended
the G\"{a}rtner-Ellis theorem and obtained a more refined large deviations result, the so-called
strong large deviation, but their results do not contain higher--order expansions.
See also \cite{Joutard}.  While some prior work (see e.g. \cite{Bordenave2007, Hult2010, ZhuI} and the references therein) have studied large deviations of certain point processes,
they also do not obtain asymptotic expansions. Higher-order large deviations expansions are only known for some special models, see e.g. \cite{BR2002}.

To overcome this difficulty, in this paper we develop a new approach that is based on
the mod-$\phi$ convergence theory studied recently in the probability literature, see e.g. \cite{Delbaen2014, Jacod2011, modphi}. In particular, the authors of \cite{modphi} used the framework of mod-$\phi$ convergence to obtain precise estimates of $\mathbb{P}(X_{n}\in t_{n} B)$  for large $t_n$ where $B$ is a Borel set,
instead of the usual asymptotic estimates for the rate of exponential decay $\lim_{n\rightarrow\infty}\frac{1}{n}\log\mathbb{P}(X_{n}\in t_{n}B)$,
for quite general sequences of random variables $(X_{n})_{n\in\mathbb{N}}$. The main idea of this framework is to look for a renormalization of the characteristic functions of $(X_{n})_{n\in\mathbb{N}}$ so that the sequence of renormalized characteristic functions converges. 
See Section~\ref{sec:extended-modphi} for mathematical details. The mod-$\phi$ framework allows one to obtain precise large deviations and refined central limit theorems  for $(X_{n})_{n\in\mathbb{N}}$ simultaneously, if one can verify that $(X_{n})_{n\in\mathbb{N}}$ converges in the mod-$\phi$ sense where $\phi$ is a (non-constant) infinitely divisible distribution. The convergence speed determines the order of asymptotic expansions of the deviation probabilities $\mathbb{P}(X_{n}\in t_{n} B)$. 

%

Specifically, our approach to derive asymptotic expansions of large deviation probabilities and tail expectations for affine point processes is as follows. 

First, we extend and sharpen the results in \cite{modphi} for quite general sequence of random variables. See Propositions~\ref{prop:modphi-lattice} and \ref{prop:modphi-nonlattice-gz}.
The paper \cite{modphi} focused on asymptotic expansions of deviation probabilities, and we extend their results to obtain asymptotic expansions of tail expectations in the context of mod-$\phi$ convergence. In addition, in the case when $\phi$ is non-lattice distributed, the paper \cite{modphi} obtained expansions of deviation probabilities of zero-th order. With additional assumptions, we obtain higher order expansions of deviation probabilities and tail expectations, by some delicate applications of
Esseen's smoothing inequality, Laplace's method and Fa\`{a} di Bruno's formula (see Appendix~\ref{sec:known}). These results hold for quite general sequence of random variables, so they are of independent interest and can be useful in other applications, e.g. establish limit theorems for other stochastic processes.

Second, we prove that for affine point processes, the sequence of random variables $(V(t))_{t>0}$ converges mod-$\phi$ exponentially fast as $t \rightarrow \infty$, see Theorem~\ref{thm:app-mod}. This step is nontrivial as one needs to identify the infinitely divisible distribution $\phi$, establish the exponential convergence and characterize the limiting function from the sequence of renormalized characteristic functions of $V(t)$ as $t$ is sent to infinity. Our proof relies on the measure-change technique in \cite{Zhang}, and a careful analysis of the affine structure of the point process as well as the ODEs governing the characteristic function of an affine point process.

With these two steps, we can then apply the general extended results in the first step to $(V(t))_{t>0}$, and
obtain explicit asymptotic expansions of large deviation probabilities and tail expectations for affine point processes. Moreover, from \cite{modphi}, the mod-$\phi$ convergence of $(V(t))_{t>0}$ we establish also implies a refined central limit theorem for affine point processes (see Theorem~\ref{CLTThm}), which extends the central limit theorem in \cite{Zhang}.  While we have limited the discussions to affine point processes in this paper due to their analytical tractability, we remark that the tools developed in this paper can be potentially used in other settings as well.

Our asymptotic analysis via the mod-$\phi$ convergence theory sheds some new insight on the long term behavior of affine point processes. 
Consider the fluctuation of $V(t)$ around its mean when time $t$ is large. We find that for 
fluctuations of smaller scales ($O(t^{1/2})$ to $o(t^{2/3})$), 
$V(t)$ behaves like a sum of $t$ i.i.d. random variables with some infinitely divisible distribution $\phi$ that we identify. 
On the other hand, for fluctuations of order $O(t)$, this is not true any more and we identify the correcting factor (the function $\psi$ in Theorem~\ref{thm:app-mod}).


The rest of the paper is organized as follows. Section~\ref{sec:app} introduces affine point processes and existing large-time asymptotics results in \cite{Zhang}. Section~\ref{sec:main} presents our main results. In Section~\ref{sec:extended-modphi}, we review the mod-$\phi$ convergence framework in \cite{modphi} and present related new results for quite general sequence of random variables.  In Section~\ref{sec:mod-app}, we establish the mod-$\phi$ convergence of $(V(t))_{t>0}$ and prove the main results in Section~\ref{sec:main}. Section~\ref{sec:numerical} presents numerical studies. Finally, proofs of other technical results and the computations of expansion coefficients for affine point processes are collected in the Appendix.

\section{Affine point processes and existing results} \label{sec:app}
To make the paper self-contained, we follow \cite{Zhang} to
introduce affine point processes in this section and review their results on large-time asymptotics of such point processes.

We fix a complete probability space $(\Omega, \mathbb{P}, \mathbb{F} )$ and a filtration $\{\mathbb{F}_t : t \ge 0\}$ satisfying the usual
conditions of right continuity and completeness (see, e.g. \cite{Karatzas2012}). We write $\mathbb{R}^d_+=\{y \in \mathbb{R}^d: y_i \ge 0, i=1, \ldots, d\}$.
Let $W = (W (t) : t \ge 0)$ be a standard $d$-dimensional Brownian motion. Let $X=(X_{1},\ldots,X_{d})$ be an affine jump diffusion satisfying the stochastic differential equation:
\begin{equation} \label{eq:SDE-affine}
dX(t)=\mu(X(t))dt+\sigma(X(t))dW(t)
+\sum_{i=1}^{n}\gamma_{i}\int_{\mathbb{R}_{+}}zN_{i}(dt,dz),
\end{equation}
with $X(0)=x_{0}$, where the drift and volatility functions are given by
\begin{align*}
&\mu(x)=b-\beta x,\quad b\in\mathbb{R}^{d},\quad\beta\in\mathbb{R}^{d\times d},
\\
&\sigma(x)\sigma(x)^{T}=a+\sum_{j=1}^{d}\alpha^{j}x_{j},
\quad
a\in\mathbb{R}^{d\times d},
\quad
\alpha^{j}\in\mathbb{R}^{d\times d},
\quad
j=1,\ldots,d.
\end{align*}
Here $\gamma_{i}\in\mathbb{R}^{d}$ and $N_{i}(dt,dz)$ is a random counting measure
on $[0, \infty) \times\mathbb{R}_{+}$ with compensator measure
$\Lambda_{i}(X(t))dt\varphi_{i}(dz)$, where $\varphi_{i}$ is a probability measure on $\mathbb{R}_{+}$
and
\begin{equation*}
\Lambda_{i}(x)=\lambda_{i}+\sum_{j=1}^{d}\kappa_{i,j}x_{j}, \quad \text{for $i=1, \ldots, n$.}
\end{equation*}
We use $Z_i$ to denote a random variable having distribution $\varphi_i$. An $n-$dimensional affine point process $L=(L_{1},\ldots,L_{n})$ is given by
\begin{equation*}
L_{i}(t):=\int_{0}^{t}\int_{\mathbb{R}_{+}}zN_{i}(ds,dz).
\end{equation*}
We note that the term $\sum_{i=1}^{n}\gamma_{i}\int_{\mathbb{R}_{+}}zN_{i}(dt,dz)$ in \eqref{eq:SDE-affine} introduces self- and mutual-excitation into $L$, i.e., the timing and marks of past events of type $i$ will directly impact the intensity and hence the future evolution of the point process.
Without this term, such effects are absent. In addition, when $\sigma(x)=0$ in \eqref{eq:SDE-affine}, the affine point process reduces to its special case Markovian Hawkes process in \cite{Hawkes}.
Further examples of affine point processes will be illustrated in Section~\ref{sec:numerical}.

The study \cite{Zhang} obtained the large-time asymptotics (See Theorems~\ref{thm:zhang1} and \ref{thm:zhang2} below) of
\begin{equation} \label{eq:Vt}
V(t):=\sum_{i=1}^{n}L_{i}(t).
\end{equation}
We aim to sharpen their results. To this end,
we follow \cite{Zhang} to impose the following assumption throughout the paper.
We use the following notations: $\textbf{A}_{I, J}=(\textbf{A}_{ij}: i \in I, j \in J)$
for a matrix $\textbf{A}$ and two index sets $I$ and $J$,  $\textbf{A}^T$ denotes the transpose,
and $Id(i)$ denotes a matrix with all entries equal to $0$ except the $i$-th diagonal entry, which is $1$.

\begin{assumption} \label{assump1}
(I) There exist index sets $I=\{1, \ldots, m\}$ and $J=\{m+1, \ldots, d\}$ such that \\
(1) $a$ is a symmetric positive semi-definite matrix with $a_{I,I}=0$.\\
(2) $\alpha^i$ is symmetric positive semi-definite and $\alpha_{I,I}^{i}= \alpha_{i,i}^{i} Id(i)$ for each $i \in I$; $\alpha^i=0$ for $i \in J$.\\
(3) $b \in \mathbb{R}^m_{+} \times \mathbb{R}^{d-m}$.\\
(4) $\beta_{I,J}=0$ and $\beta_{I,I}$ has nonpositive off-diagonal elements.\\
(5) $\lambda = (\lambda_i) \in \mathbb{R}^n_{+}$, $\kappa \in \mathbb{R}^{n \times d}$ with $\kappa _{i,J}=0$ for $i=1, \ldots, n$.\\
(6) $\gamma_i \in \mathbb{R}^m_{+} \times \mathbb{R}^{d-m}$, for $i=1, \ldots, n$.\\
(II) $\alpha^i_{i,i}>0$ and $b_i>0$ for each $i=1, \ldots,m$; $\lambda_{i}+\sum_{j=1}^{m}\kappa_{i,j}> 0$ for each $i=1, \ldots,n$. \\
(III) $\beta-\sum_{i=1}^{n}\mathbb{E}[Z_i]\gamma_{i}\kappa_{i}^{T}$ is positive stable, where $\kappa_i^T$ is the $i$-th row of ${\kappa}=(\kappa_{i,j})$, $i=1, \ldots, n$.
\end{assumption}
This assumption on the parameters ($\alpha$, $a$, $b$, $\beta$, $\lambda$, $\kappa$, $\gamma$) of the SDE in \eqref{eq:SDE-affine}
essentially ensures that the affine point process is properly defined and $X(\cdot)$ in \eqref{eq:SDE-affine} is ergodic.
We refer the readers to \cite{Zhang} for further discussions on this assumption.

We next summarize two main mathematical results in \cite{Zhang}. The first result is  a central limit theorem for $V(t)$ as $t \rightarrow \infty$.

\begin{theorem}[Theorem 1 in \cite{Zhang}] \label{thm:zhang1}
Assume $\mathbb{E}[(Z_i)^{(2+\epsilon)}]<\infty$ for some $\epsilon>0$ for all $i=1, \ldots, n$. Under Assumption~\ref{assump1},
we have as $t \rightarrow \infty$,
\begin{equation*}
\frac{V(t)-rt}{\sqrt{t}}\rightarrow N(0,\sigma^{2}),
\end{equation*}
in distribution, where
\begin{align}
&r=\mathcal{A}^{T}b+\sum_{i=1}^{n}\lambda_{i}\mathbb{E}[Z_i](1+\mathcal{A}^{T}\gamma_{i}),
\nonumber 
\\
&\sigma^{2}=\mathcal{A}^{T}a\mathcal{A}+\mathcal{C}^{T}\lambda+(\mathcal{A}^{T}\alpha\mathcal{A}+\mathcal{C}^{T}\kappa)\mathcal{B},
\nonumber
\\
&\mathcal{A}^{T}=\left(\sum_{i=1}^{n}\mathbb{E}[Z_i]\kappa_{i}^{T}\right)
\left(\beta-\sum_{i=1}^{n}\mathbb{E}[Z_i]\gamma_{i}\kappa_{i}^{T}\right)^{-1},
\nonumber 
\\
&\mathcal{B}=\left(\beta-\sum_{i=1}^{n}\mathbb{E}[Z_i]\gamma_{i}\kappa_{i}^{T}\right)^{-1}\left(b+\sum_{i=1}^{n}\lambda_{i}\mathbb{E}[Z_i]\gamma_{i}\right),
\nonumber
\\
&\mathcal{C}_{i}=(1+\mathcal{A}^{T}\gamma_{i})^{2}\mathbb{E}(Z_i)^{2},\qquad i=1,2,\ldots,n.
\nonumber
\end{align}
\end{theorem}

The second result is a large deviation principle for $V(t)$ as $t \rightarrow \infty$.

\begin{theorem}[Theorem 2 in \cite{Zhang}]\label{thm:zhang2}
 Assume $R>r$ and $\sup \{ \theta \in \mathbb{R}: \mathbb{E}[e^{\theta Z_i}]< \infty \} >0$ for each $i=1, \ldots, n$.
Under Assumption~\ref{assump1},
we have
\begin{equation*}
\lim_{t\rightarrow\infty}\frac{1}{t}\log\mathbb{P}(V(t)\geq Rt)=-\mathcal{I}(R),
\end{equation*}
where $\mathcal{I}(R)=\sup_{\theta\in\mathbb{R}}\{\theta R-\eta(\theta)\}$,
and
\begin{equation} \label{eq:eta-phi}
\eta(\theta) =u^{\ast}(\theta)^{T}b
+\sum_{i=1}^{d}\lambda_{i}\left(\mathbb{E}\left[e^{(\theta+u^{\ast}(\theta)^{T}\gamma_{i})Z_i}\right]-1\right),
\end{equation}
where
$u^{\ast}(\theta):\mathbb{R}\rightarrow\mathbb{R}^{n}$ is the implicit function defined
as the unique solution branch with $u^{\ast}(0)=0$ of the system of nonlinear equations:
\begin{equation} \label{eq:u-ast}
\sum_{i=1}^{d}u_{i}\beta_{i,j}-\frac{1}{2}u^{T}\alpha^{j}u-\sum_{i=1}^{n}
\left(\mathbb{E}\left[e^{(\theta+u^{T}\gamma_{i})Z_i}\right]-1\right)\kappa_{i,j}=0,
\qquad
j=1,2,\ldots,d.
\end{equation}
\end{theorem}

{Note that there exists some $\theta_{c}>0$ so that $u^{\ast}(\theta)$ and $\eta(\theta)$ in Theorem \ref{thm:zhang2} above are well-defined
for any $\theta\leq\theta_{c}$. For the definition of $\theta_{c}$ and related discussions,
we refer to Section 4.3 in \cite{Zhang}.}

Before we present our results, we remark that \cite{Zhang} also considered the limit theorems for weighted combinations of $L_i(t)$ and the left tail of $V(t)$ as $t \rightarrow \infty$.
For the simplicity of the presentation, in this paper we restrict our discussions to the refinement of the above two results (Theorems~\ref{thm:zhang1} and \ref{thm:zhang2}) only, although similar refinements
can also be obtained for weighted combinations of $L_{i}(t)$ and the left tail of $V(t)$.


\section{Main results: refined large-time asymptotics} \label{sec:main}
In this section, we present our main results on refinements of large--time asymptotics of affine point processes.
Recall that for two real-valued functions $f_1, f_2$, we write $f_1=O(f_2)$ as $t \rightarrow \infty$ if there are constants $c_1$ and $c_2 > 0$ such that $|f_1(t)| \le c_1 |f_2(t)|$ whenever $t>c_2$.
We write $f_1=o(f_2)$ as $t \rightarrow \infty$ if $f_1(t)/f_2(t) \rightarrow 0$ as $t\rightarrow\infty$.

Our first result is a precise large deviation result for $V(t)$ when $t$ goes to infinity, which extends Theorem~\ref{thm:zhang2} (Theorem 2 in \cite{Zhang}). Recall the quantities $r, \mathcal{I}(R), \eta$ given in Theorem~\ref{thm:zhang2}. 


\begin{theorem}\label{thm:preciseLDP}

Let $R> r$ and $h \in \mathbb{R}$ defined by $\eta'(h)=R$. Under the same assumptions as in Theorem~\ref{thm:zhang2}, we have the following.

\begin{itemize}


\item [(1)] If the random variables
$\sum_{i=1}^{n}\sum_{j=1}^{n_{i}}Z_{ij}$, where $n_{i}\in\mathbb{N}\cup\{0\}$
and $Z_{ij}$ are i.i.d. distributed as $Z_{i}$,
are supported on $\mathbf{b}\mathbb{N}\cup\{0\}$ for some parameter $\mathbf{b}>0$,
then there exist constants $(c_k)_{k=0}^{\infty}$ such that
for $Rt\in \mathbf{b} \mathbb{N}$, as $t \rightarrow \infty,$
\begin{equation} \label{eq:LtRt}
\mathbb{P}(V(t) \geq R t)
=\frac{e^{-t \mathcal{I}(R) }}{\sqrt{2\pi t \eta''(h)}} \cdot \left(c_0 + \sum_{k=1}^{v}\frac{c_k}{t^k}
+ O\left(\frac{1}{t^{v}}\right) \right),
\end{equation}
where $v$ is any positive integer.
More generally, assume $g:\mathbb{R}_{+} \rightarrow\mathbb{R}$
satisfies $|g(x)|\leq \bar ae^{\bar h x}$ for some $\bar h<h$ and $\bar a>0$,
then there exist constants $(\hat c_k)_{k=0}^{\infty}$ such that
for any $Rt\in \mathbf{b} \mathbb{N}$, as $t \rightarrow \infty,$ we have
\begin{equation} \label{eq:ES}
\mathbb{E}\left[g(V(t)- R t) \cdot 1_{V(t) \geq R t} \right]
=  \frac{e^{-t \mathcal{I}(R) }}{\sqrt{2\pi t \eta''(h)}} \cdot \left( {\hat c_0 + \sum_{k=1}^{v}\frac{\hat c_k}{t^k}  + O\left(\frac{1}{t^{v+1}} \right)  } \right).
\end{equation}

\item [(2)] Otherwise, there exist constants $(d_k)_{k=0}^{\infty}$ such that as $t \rightarrow \infty,$
\begin{equation} \label{eq:LtRt2}
\mathbb{P}(V(t) \geq R t)
=\frac{e^{-t \mathcal{I}(R) }}{\sqrt{2\pi t \eta''(h)}} \cdot \left(d_0 + \sum_{k=1}^{v}\frac{d_k}{t^k}
+ o \left(\frac{1}{t^{v}}\right) \right),
\end{equation}
where $v$ is any positive integer.
{More generally, assume that $g:\mathbb{R}_{+}\rightarrow\mathbb{R}$
admits the expansion $g(x)=\sum_{k=0}^{\infty}g_{k}x^{k+\Delta}$,
where $\Delta\in[0,1)$
and there exist some $\bar{a}>0$ and $0<\bar{h}<h$, such that $g_{k}\leq\bar{a}\frac{\bar{h}^{k}}{k!}$
for every $k\in\mathbb{N}$,}
then there exist constants $(\hat d_k)_{k=0}^{\infty}$ such that as $t \rightarrow \infty,$ 
\begin{equation} \label{eq:ES2}
\mathbb{E}\left[g(V(t)- R t) \cdot 1_{V(t) \geq R t} \right]
=\frac{e^{-t \mathcal{I}(R) }}{\sqrt{2\pi t \eta''(h)}} \cdot \left( {\hat{d}_{0} + \sum_{k=1}^{v}\frac{\hat d_k}{t^k}+o \left(\frac{1}{t^{v}}\right)} \right).
\end{equation}

\end{itemize}
\end{theorem}

The coefficients $(c_k, \hat c_k$, $d_k, \hat d_k)_{k=0}^{\infty}$ in the above result have explicit formulas (see Propositions~\ref{prop:modphi-lattice} and \ref{prop:modphi-nonlattice-gz}, and Appendix~\ref{sec:affine-coeff}), and can be numerically computed.
Hence, based on Theorem~\ref{thm:preciseLDP},
we can develop approximations for large deviation tail probabilities and tail expectations for $V(t)$
by truncations of the asymptotic expansions.
Numerical results on the quality of approximations will be illustrated in Section~\ref{sec:numerical}.

Our next result is a refined central limit theorem (CLT) for $V(t)$, which improves Theorem~\ref{thm:zhang1} (i.e. Theorem 1 in \cite{Zhang})
{under a stronger assumption}.

\begin{theorem}[Refined CLT]\label{CLTThm}
Let $y=o(t^{1/6})$. Assume the random variable $V(t)$ is either lattice distributed or it has a non-lattice law that is absolutely continuous with respect to Lebesgue measure. Under the same assumptions as in Theorem~\ref{thm:zhang2},
we have as $t \rightarrow \infty$,
\begin{equation*}
\mathbb{P}\left(V(t)\geq rt+\sigma y\sqrt{t}\right)
=\int_{y}^{\infty}\frac{e^{-\frac{u^{2}}{2}}}{\sqrt{2\pi}}du
\cdot(1+o(1)),
\end{equation*}
where $r$ and $\sigma$ are given in Theorem~\ref{thm:zhang1}.
\end{theorem}

To prove Theorems~\ref{thm:preciseLDP} and \ref{CLTThm}, we use an approach that is based on the mod-$\phi$ convergence framework which we discuss in the next section. The stronger assumption needed in Theorem~\ref{CLTThm} is required in using this approach. On the other hand,
in \cite{Zhang}, one needs extra effort to prove the large deviation principle in addition to the central limit theorem, while the mod-$\phi$ convergence framework leads to a unified approach to establish both precise large deviations and refined CLT.
We give the details of the proofs of Theorems~\ref{thm:preciseLDP}  and \ref{CLTThm} in Section~\ref{sec:mod-app}.

\section{Mod-$\phi$ convergence framework and some new results}\label{sec:extended-modphi}

In this section, we briefly review the mod-$\phi$ convergence framework in \cite{modphi}, and present some related new results ({Propositions~\ref{prop:modphi-lattice} and \ref{prop:modphi-nonlattice-gz}}) for \textit{quite general sequence of random variables}.
These new results extend the results on precise estimates of large deviations probabilities in \cite{modphi}, and they are of independent interest.

We first recall the definition of mod-$\phi$ convergence, where one considers
a renormalization of the characteristic functions of random variables.
Let $(X_{n})_{n\in\mathbb{N}}$ be a sequence of real-valued random variables
and $\mathbb{E}[e^{zX_{n}}]$ exist in a strip $\mathbb{S}_{(c,d)}:=\{z\in\mathbb{C}: c<\mathcal{R}(z)<d\}$,
with $c<d$ being extended real numbers, (i.e. we allow $c=-\infty$ and $d=+\infty$)
and $\mathcal{R}(z)$ denotes the real part of $z\in\mathbb{C}$.

\begin{definition} [Definition 1.1 in \cite{modphi}]\label{def:mod-phi}
We say that $(X_{n})_{n\in\mathbb{N}}$ converges mod-$\phi$ on $\mathbb{S}_{(c,d)}$
with parameters $(t_{n})_{n\in\mathbb{N}}$ and the limiting function $\psi$, if
 there exists a non-constant infinitely divisible distribution $\phi$
with $\int_{\mathbb{R}}e^{zx}\phi(dx)=e^{\eta(z)}$, which is well defined
on $\mathbb{S}_{(c,d)}$, and an analytic function $\psi(z)$
that does not vanish on the real part of $\mathbb{S}_{(c,d)}$
such that as $n\rightarrow\infty$, we have $t_{n}\rightarrow+\infty$, and
\begin{equation}\label{eq:mod}
e^{-t_{n}\eta(z)} \cdot \mathbb{E}\left[e^{zX_{n}}\right]\rightarrow\psi(z),
\quad \text{locally uniformly in $z\in\mathbb{S}_{(c,d)}$}.
\end{equation}
 In addition, we say that $(X_{n})_{n\in\mathbb{N}}$ converges mod-$\phi$ at speed $O((t_n)^{-v})$ (respectively, exponentially fast) if the difference of the two sides of Equation~\eqref{eq:mod} is uniformly bounded by $C_K (t_n)^{-v}$ (respectively, $C_K e^{-t \bar C_K}$) for $z$ in a compact subset $K$ of $\mathbb{S}_{(c,d)}$, for some positive constants $C_K$ and $\bar C_K$.
\end{definition}
 An informal interpretation of the convergence in \eqref{eq:mod} is that $X_n$ can be represented as the sum of $t_n$ independent copies of the (non-constant) infinitely divisible distribution $\phi$ plus a perturbation encoded in the limiting function $\psi$.
Note in the above definition, we have slightly abused the notations for convenience and the meanings should be clear from context: we use the same notation $\eta$ in the general setting here and in Sections~\ref{sec:app} and \ref{sec:main} for affine point processes; we use $n$ to denote the index of random variables $X_n$ here and use $n$ to denote the dimension of affine point processes in Section~\ref{sec:app}.

The mod-$\phi$ convergence framework allows one to obtain precise estimates for tail probabilities and one needs to consider two separate cases: lattice distributed $\phi$, and non-lattice distributed $\phi$. We next review the results in \cite{modphi}
and present our new results for these two cases. To facilitate the presentation, we write
\begin{equation}\label{eqn:Sn}
\mathcal{S}_{n} := \left\{ (m_{1},\ldots,m_{n}):  1\cdot m_{1}+2\cdot m_{2}+\cdots+n\cdot m_{n}=n, \text{with each }m_{i}\in\mathbb{N}\cup\{0\}\right\}.
\end{equation}
This set $\mathcal{S}_{n}$ appears whenever
we apply Fa\`{a} di Bruno's formula (see Lemma~\ref{lem:faadi} in Appendix~\ref{sec:Faa}) for derivatives of composite functions for $n\geq 1$,
and whenever $\mathcal{S}_{0}$ appears, we simply consider the $0$-th order derivative.

\subsection{Lattice case}  
We first discuss the case where $\phi$ is lattice distributed.
For a given $x$, define $h$ (with a slight abuse of notations) and $F(x)$ by
\begin{equation}\label{eq:hF}
\eta'(h)=x, \quad \text{and} \quad F(x)=\sup_{\theta\in\mathbb{R}}\{\theta x-\eta(\theta)\}.
\end{equation}
The following result from \cite{modphi} yields an expansion for the tail probabilities when $\phi$ has a lattice distribution, which generalizes the Buhadur--Rao theorem for sums of i.i.d random variables in large deviations theory \cite{Bahadur1960, Dembo}. Without loss of generality, as in \cite{modphi}, we assume that $X_n$'s and the infinite divisible distribution $\phi$ both take values in $\mathbb{Z}$ and $\mathbb{Z}$ is the minimal lattice for $\phi$.

\begin{theorem}[Theorem 3.4. \cite{modphi}] \label{thm:3}
Suppose $\phi$ is lattice distributed, and $(X_{n})_{n\in\mathbb{N}}$ converges
mod-$\phi$ at speed $O((t_n)^{-v})$ on a band $\mathbb{S}_{(c,d)}$ ($c<0<d$) with parameters $(t_{n})_{n\in\mathbb{N}}$ and the limiting function $\psi$. Then for $x \in (\eta'(0), \eta'(d))$ and $t_{n}x\in\mathbb{N}$, we have as $n \rightarrow \infty,$
\begin{equation*} \label{eq:expansion}
\mathbb{P}(X_{n}\geq t_{n}x)
=\frac{e^{-t_{n}F(x)}}{\sqrt{2\pi t_{n}\eta''(h)}} \left(c_0 + \frac{c_1}{t_n} + \ldots + \frac{c_{v-1}}{(t_n)^{v-1}} + O\left(\frac{1}{(t_n)^v}\right) \right),
\end{equation*}
where $c_0= \frac{\psi(h)}{1-e^{-h}}$ and $(c_{k})_{k=1}^{v-1}$ can be computed as follows:
\begin{align*}
c_{k}
&=\sum_{m+\ell+n=2k}
\sum_{q=0}^{\infty}e^{-qh}\frac{(-q)^{m}}{m!}
\cdot
\frac{\psi^{(\ell)}(h)}{\ell!}
\\
&\qquad
\cdot
\sum_{\mathcal{S}_{n}}\frac{(-1)^{m_{1}+\cdots+m_{n}}}{m_{1}!1!^{m_{1}}m_{2}!2!^{m_{2}}\cdots m_{n}!n!^{m_{n}}}
\\
&\qquad
\cdot
\prod_{j=1}^{n}\left(\frac{1}{\eta''(h)}\frac{\eta^{(j+2)}(h)}{(j+2)(j+1)}\right)^{m_{j}}
\frac{(-1)^{k}(2(k+m_{1}+\cdots+m_{n})-1)!!}{(\eta''(h))^{k}},
\end{align*}
where $\mathcal{S}_{n}$ is defined in \eqref{eqn:Sn}.
\end{theorem}

We extend the above theorem to obtain precise estimates of tail expectations. The proof of the following result is given in Appendix~\ref{appendixA}.

\begin{proposition}\label{prop:modphi-lattice}
Suppose the assumptions in Theorem~\ref{thm:3} hold.
In addition, assume $g:\mathbb{R}_{+} \rightarrow\mathbb{R}$
satisfies $|g(x)|\leq \bar a e^{\bar h x}$ for some $\bar h <h$ and $\bar a>0$.
Then as $n \rightarrow \infty,$ we have
\begin{equation*}
\mathbb{E}[g(X_{n}-t_{n}x)1_{X_{n}\geq t_{n}x}]
=\frac{e^{-t_{n}F(x)}}{\sqrt{2\pi t_{n}\eta''(h)}}   \left(\hat c_0 + \frac{\hat c_1}{t_n} + \ldots + \frac{\hat c_{v-1}}{t_n^{v-1}} + O\left(\frac{1}{t_n^v}\right) \right),
\end{equation*}
where $h$ and $F$ are defined in \eqref{eq:hF}. Here, $\hat c_0= \sum_{k=0}^{\infty}g(k)e^{-kh} \cdot \psi(h)$, and
$(\hat{c}_{k})_{k=1}^{v-1}$ can be computed using $g$ and the higher order
derivatives of $\eta$ and $\psi$ at $h$:
\begin{align} \label{eq:hatck}
\hat{c}_{k}
&=\sum_{m+\ell+n=2k}
\sum_{q=0}^{\infty}g(q)e^{-qh}\frac{(-q)^{m}}{m!}
\cdot
\frac{\psi^{(\ell)}(h)}{\ell!}
\nonumber \\
&\qquad
\cdot
\sum_{\mathcal{S}_{n}}\frac{(-1)^{m_{1}+\cdots+m_{n}}}{m_{1}!1!^{m_{1}}m_{2}!2!^{m_{2}}\cdots m_{n}!n!^{m_{n}}}
\nonumber \\
&\qquad
\cdot
\prod_{j=1}^{n}\left(\frac{1}{\eta''(h)}\frac{\eta^{(j+2)}(h)}{(j+2)(j+1)}\right)^{m_{j}}
\frac{(-1)^{k}(2(k+m_{1}+\cdots+m_{n})-1)!!}{(\eta''(h))^{k}},
\end{align}
where $\mathcal{S}_{n}$ is defined in \eqref{eqn:Sn}.
\end{proposition}

\begin{remark} \label{remark1}
From the above two results, one can readily find that
\begin{align*}
c_{1}
&=-\frac{1}{2}\psi(h)\frac{e^{-h}+e^{-2h}}{(1-e^{-h})^{3}}\frac{1}{\eta''(h)}
-\frac{1}{2}\frac{1}{1-e^{-h}}\psi''(h)\frac{1}{\eta''(h)}
\\
&\qquad\qquad
+\frac{1}{1-e^{-h}}
\psi(h)\left[\frac{1}{8}\eta^{(4)}(h)\frac{1}{(\eta''(h))^{2}}
-\frac{5}{24}(\eta^{(3)}(h))^{2}\frac{1}{(\eta''(h))^{3}}\right]
\\
&\qquad
+\frac{e^{-h}}{(1-e^{-h})^{2}}\psi'(h)\frac{1}{\eta''(h)}
+\frac{1}{2}\frac{1}{1-e^{-h}}\psi'(h)\eta^{(3)}(h)\frac{1}{(\eta''(h))^{2}}
\\
&\qquad\qquad
-\frac{1}{2}\frac{e^{-h}\psi(h)}{(1-e^{-h})^{2}}\eta^{(3)}(h)\frac{1}{(\eta''(h))^{2}},
\end{align*}
and
\begin{align*}\label{eq:b1}
\hat c_{1}
&=-\frac{1}{2}\psi(h)\sum_{k=0}^{\infty}g(k)e^{-kh}k^{2}\frac{1}{\eta''(h)}
-\frac{1}{2}\sum_{k=0}^{\infty}g(k)e^{-kh}\psi''(h)\frac{1}{\eta''(h)}
\\
&\qquad\qquad
+\sum_{k=0}^{\infty}g(k)e^{-kh}
\psi(h)\left[\frac{1}{8}\eta^{(4)}(h)\frac{1}{(\eta''(h))^{2}}
-\frac{5}{24}(\eta^{(3)}(h))^{2}\frac{1}{(\eta''(h))^{3}}\right]
\nonumber  \\
&\qquad \qquad
+\sum_{k=0}^{\infty}g(k)e^{-kh}k\psi'(h)\frac{1}{\eta''(h)}
+\frac{1}{2}\sum_{k=0}^{\infty}g(k)e^{-kh}\psi'(h)\eta^{(3)}(h)\frac{1}{(\eta''(h))^{2}}
\nonumber  \\
&\qquad\qquad
-\frac{1}{2}\sum_{k=0}^{\infty}g(k)e^{-kh}k \psi(h)\eta^{(3)}(h)\frac{1}{(\eta''(h))^{2}}. \nonumber
\end{align*}
We will use and compute these coefficients in numerical experiments in Section~\ref{sec:numerical}.
\end{remark}

\subsection{Non-lattice case}
We next discuss the case where $\phi$ is non-lattice distributed. We first present the result from \cite{modphi}.

\begin{theorem}[Theorem 4.3. \cite{modphi}]
Suppose $\phi$ is non-lattice, and $(X_{n})_{n\in\mathbb{N}}$ converges
mod-$\phi$ with parameters $(t_{n})_{n\in\mathbb{N}}$ and the limiting function $\psi$ on a band $\mathbb{S}_{(c,d)}$ with $c<0<d$ .
If $x\in(\eta'(0),\eta'(d))$, then
\begin{equation*}
\mathbb{P}(X_{n}\geq t_{n}x)
=\frac{e^{-t_{n}F(x)}}{\sqrt{2\pi t_{n}\eta''(h)}} \cdot \left(\frac{\psi(h)}{h}+o(1) \right),
\end{equation*}
where $h$ is defined via $\eta'(h)=x$ and $F(x):=\sup_{\theta\in\mathbb{R}}\{\theta x-\eta(\theta)\}$.
\end{theorem}

In Proposition~\ref{prop:modphi-nonlattice-gz} below,
we sharpen the above result and extend it to obtain precise estimates for tail expectations under additional assumptions.
The proof (see Appendix~\ref{appendixA}) relies on some delicate applications of
Esseen's smoothing inequality, Laplace's method and Fa\`{a} di Bruno's formula
(see Appendix~\ref{sec:known}).

\begin{proposition}\label{prop:modphi-nonlattice-gz}
Suppose $\phi$ is non-lattice distributed, and $(X_{n})_{n\in\mathbb{N}}$ converges
mod-$\phi$ with parameters $(t_{n})_{n\in\mathbb{N}}$ and the limiting function $\psi$ at the speed $O((t_n)^{-v})$
on a band $\mathbb{S}_{(c,d)}$ with $c<0<d$.
Assume that $x\in(\eta'(0),\eta'(d))$.

(i) We have as $t_{n}\rightarrow\infty$,  
\begin{equation*}
\mathbb{P}(X_{n}\geq t_{n} x)
=\frac{e^{-t_{n}F(x)}}{\sqrt{2\pi t_{n}\eta''(h)}}
\left[d_{0}+\frac{d_{1}}{t_{n}}+\frac{d_{2}}{t_{n}^{2}}+\cdots
+\frac{d_{v-1}}{t_{n}^{v-1}}+o\left(t_{n}^{-v+1}\right)\right].
\end{equation*}
Here, $(d_{k})_{k=0}^{v-1}$ are given by
\begin{align*}
d_{k}
&=\frac{1}{h}
\cdot\sum_{p=0}^{2k}
\sum_{\ell=0}^{p}\frac{\psi^{(p-\ell)}(h)}{(p-\ell)!}
\\
&\qquad
\cdot\sum_{\mathcal{S}_{\ell}}\frac{1}{m_{1}!1!^{m_{1}}m_{2}!2!^{m_{2}}\cdots m_{\ell}!\ell!^{m_{\ell}}}
\cdot\prod_{j=1}^{\ell}\left(\frac{1}{\eta''(h)}\frac{\eta^{(j+2)}(h)}{(j+2)(j+1)}\right)^{m_{j}}
\frac{1}{(\eta''(h))^{p/2}}
\\
&\qquad\cdot
\sum_{m=0}^{\lfloor\frac{p}{2}+m_{1}+\cdots+m_{\ell}\rfloor}
\frac{2^{m}(-1)^{m}(p+2(m_{1}+\cdots+m_{\ell}))!}{m!(p+2(m_{1}+\cdots+m_{\ell})-2m)!}
\\
&\qquad\qquad\qquad\cdot
\frac{\Gamma(k+(m_{1}+\cdots+m_{\ell})-m+1)a_{k+m-p-(m_{1}+\cdots+m_{\ell})}}{(h^{2}\eta''(h))^{-\frac{p}{2}+k}},
\end{align*}
where $\mathcal{S}_{\ell}$ is defined in \eqref{eqn:Sn} and $(a_{k})_{k=0}^{\infty}$ is defined recursively as:
\begin{equation*}
a_{k}=(-1)^{k}-\sum_{j=0}^{k-1}\frac{(j+s)_{k-j}}{(k-j)!2^{k-j}}a_{j},
\end{equation*}
and $(\cdot)_{i}$ is the Pochhammer symbol
\footnote{Pochhammer symbol is defined as $(x)_{n}=\Gamma(x+n)/\Gamma(x)$ where $\Gamma(\cdot)$ is the Gamma function.}.

(ii) Furthermore, assume that $g:\mathbb{R}_{+}\rightarrow\mathbb{R}$
admits the expansion $g(x)=\sum_{k=0}^{\infty}g_{k}x^{k+\Delta}$ where $\Delta\in[0,1)$
{and there exist some $\bar{a}>0$ and $0<\bar{h}<h$ so that $g_{k}\leq\bar{a}\frac{\bar{h}^{k}}{k!}$
for every $k\in\mathbb{N}$.}
Then we have
\begin{equation*}
\mathbb{E}\left[g(X_{n}-t_{n}x)1_{X_{n}\geq t_{n}x}\right]
=\frac{e^{-t_{n}F(x)}}{\sqrt{2\pi}\sqrt{t_{n}\eta''(h)}}
\left[\hat{d}_{0}+\frac{\hat{d}_{1}}{t_{n}}+\frac{\hat{d}_{2}}{t_{n}^{2}}+\cdots
+\frac{\hat{d}_{n-1}}{t_{n}^{v-1}}+o\left(t_{n}^{-v+1}\right)\right],
\end{equation*}
as $t_{n}\rightarrow\infty$, where $(\hat{d}_{k})_{k=0}^{v-1}$ are given by
\begin{align*}
\hat{d}_{k}
&=\sum_{q=0}^{\infty}g_{q}
\cdot
\frac{1}{h^{\Delta+q+1}}
\cdot\sum_{p=0}^{2k}
\sum_{\ell=0}^{p}\frac{\psi^{(p-\ell)}(h)}{(p-\ell)!}
\\
&\qquad
\cdot\sum_{\mathcal{S}_{\ell}}\frac{1}{m_{1}!1!^{m_{1}}m_{2}!2!^{m_{2}}\cdots m_{\ell}!\ell!^{m_{\ell}}}
\cdot\prod_{j=1}^{\ell}\left(\frac{1}{\eta''(h)}\frac{\eta^{(j+2)}(h)}{(j+2)(j+1)}\right)^{m_{j}}
\frac{1}{(\eta''(h))^{p/2}}
\\
&\qquad\cdot
\sum_{m=0}^{\lfloor\frac{p}{2}+m_{1}+\cdots+m_{\ell}\rfloor}
\frac{2^{m}(-1)^{m}(p+2(m_{1}+\cdots+m_{\ell}))!}{m!(p+2(m_{1}+\cdots+m_{\ell})-2m)!}
\\
&\qquad\qquad\qquad\cdot
\frac{\Gamma(k+\Delta+q+(m_{1}+\cdots+m_{\ell})-m+1)a_{k+m-p-(m_{1}+\cdots+m_{\ell})}}{(h^{2}\eta''(h))^{-\frac{p}{2}+k}}.
\end{align*}
\end{proposition}

\begin{remark}
From Part (i) of the above result, one can find that
$d_{0}=\frac{\psi(h)}{h},$
and
\begin{align*}
d_{1}&=
-\frac{\psi(h)}{h^{3}\eta''(h)}
+\psi(h)\left(\frac{\psi'(h)}{\psi(h)}-\frac{\eta'''(h)}{\eta''(h)}\right)
\frac{1}{h^{2}\eta''(h)}
\\
&\qquad
+\frac{\psi(h)}{h}\left[\frac{1}{8}\frac{\eta^{(4)}(h)}{(\eta''(h))^{2}}
-\frac{5}{24}\frac{(\eta^{(3)}(h))^{2}}{(\eta''(h))^{3}}
-\frac{1}{2}\frac{\psi''(h)}{\psi(h)\eta''(h)}
+\frac{1}{2}\frac{\eta^{(3)}(h)\psi'(h)}{\psi(h)(\eta''(h))^{2}}\right].
\end{align*}
Similarly, from Part (ii) of the above result, one can check that when $g(x) = x^{\gamma}$ for $\gamma \ge 0$, we have
\begin{equation*}
\hat{d}_{0}=\frac{\Gamma(\gamma+1)\psi(h)}{h^{\gamma+1}},
\end{equation*}
and
\begin{align*}
\hat{d}_{1}&=
-\frac{\psi(h)}{h^{\gamma+1}}\frac{\Gamma(\gamma+2)(1+\frac{\gamma}{2})}{h^{2}\eta''(h)}
+\frac{\psi(h)}{h^{\gamma+1}}\left(\frac{\psi'(h)}{\psi(h)}-\frac{\eta'''(h)}{\eta''(h)}\right)
\frac{\Gamma(\gamma+2)}{h\eta''(h)}
\\
&\qquad
+\frac{\psi(h)}{h^{\gamma+1}}\Gamma(\gamma+1)\left[\frac{1}{8}\frac{\eta^{(4)}(h)}{(\eta''(h))^{2}}
-\frac{5}{24}\frac{(\eta^{(3)}(h))^{2}}{(\eta''(h))^{3}}
-\frac{1}{2}\frac{\psi''(h)}{\psi(h)\eta''(h)}
+\frac{1}{2}\frac{\eta^{(3)}(h)\psi'(h)}{\psi(h)(\eta''(h))^{2}}\right].
\end{align*}
We will use and compute these coefficients in numerical experiments in Section~\ref{sec:numerical}.
\end{remark}


\section{Mod-$\phi$ convergence for affine point processes and proofs of results in Section~\ref{sec:main}}\label{sec:mod-app}
In this section, we prove the main results in Section~\ref{sec:main}. The key step is to establish the mod-$\phi$ convergence of the sequence of random variables $(V(t))_{t>0}$ in \eqref{eq:Vt}. We summarize the result below in Section~\ref{sec:convg-V} and present its proof in Section~\ref{thm:modphi-app}.

\subsection{Mod-$\phi$ convergence of $(V(t))_{t>0}$} \label{sec:convg-V}

Recall from Theorem~\ref{thm:zhang2} that $\eta(\theta)$ and $u^*(\theta)$ are well defined for a real number $\theta \le \theta_c$ with $\theta_c>0$.  One can naturally extends their definitions to a complex number $\theta \in \mathbb{C}$ with the real part $\mathcal{R}(\theta)\leq\theta_{c}$, so that they fit in the mod-$\phi$ convergence framework (Definition~\ref{def:mod-phi}) which deals with characteristic functions. Also recall the parameters of an affine point process introduced in Section~\ref{sec:app}. 

\begin{theorem} \label{thm:app-mod}

%
%
%
%
%
%

The sequence of random variables $(V(t))_{t>0}$ converges mod-$\phi$ exponentially fast as $t \rightarrow \infty$
along any lattice
with limiting function $\psi$, where $e^{\eta(\theta)}=\int_{\mathbb{R}} e^{\theta x} \phi(dx)$,
\begin{equation}\label{psiDefn}
\psi(\theta)=e^{u^{\ast}(\theta)^{T}x_{0}+B(\infty;\theta,u^{\ast}(\theta))},
\end{equation}
and $\psi$ is analytic for $\mathcal{R}(\theta)<\theta_{c}$.
Here,
\begin{align}
B(t;\theta,\delta)&=\int_{0}^{t}b^{T}A(s;\theta,\delta)ds+\int_{0}^{t}\frac{1}{2}A^{T}aAds
\nonumber \\
&\qquad\qquad
+\sum_{i=1}^{n}\lambda_{i}\int_{0}^{t}\int_{\mathbb{R}_{+}}\left(e^{A^{T}\gamma_{i}z}-1\right)e^{(\theta+\delta^{T}\gamma_{i})z}\varphi_{i}(dz)ds, \label{eq:B}
\end{align}
is a scalar function, where
$A(t;\theta,\delta)=(A_{1}(t;\theta,\delta),\cdots,A_{n}(t;\theta,\delta))$ is a matrix function,
with $A(0;\theta,\delta)=-\delta$, $B(0;\theta,\delta)=0$ and
\begin{equation} \label{eq:A}
\frac{d}{dt}A_{j}(t;\theta,\delta)
=-\sum_{i=1}^{n}A_{i}\beta_{i,j}^{\ast}+\frac{1}{2}A^{T}\alpha^{j}A+\sum_{i=1}^{n}
\int_{\mathbb{R}_{+}}\left(e^{A^{T}\gamma_{i}z}-1\right)e^{(\theta+\delta^{T}\gamma_{i})z}\varphi_{i}(dz)\kappa_{i,j},
\end{equation}
for $j=1,2,\ldots,n$, where $\beta^{\ast}$ is defined in \eqref{beta:star}.
\end{theorem}

Several remarks are in order. First, the infinite divisible distribution $\phi$ will be identified and given in Section~\ref{Sec:Infinite}.
Second, note that $(V(t))_{t>0}$ is a continuous-time process, so to be consistent with Definition~\ref{def:mod-phi}, we have stated
the mod-$\phi$ convergence of $V(t)$ along any lattice (i.e. $t=b^{\ast}k$ for any $b^{\ast}>0$ and $k \in \mathbb{N}$ with $k \rightarrow \infty$) in the above result. Finally, we use the notation $B(\infty;\theta,u^{\ast}(\theta))$ in \eqref{psiDefn} to denote $\lim_{t \rightarrow \infty} B(t;\theta,u^{\ast}(\theta))$ which will be shown to exist in the proof in Section~\ref{thm:modphi-app}.



Relying on Theorem~\ref{thm:app-mod}, we are ready to prove the results in Section~\ref{sec:main} in the next section.


\subsection{Proofs of Theorems~\ref{thm:preciseLDP} and \ref{CLTThm} in Section~\ref{sec:main}}


\begin{proof}
We first prove Theorem~\ref{thm:preciseLDP}. For Part (1), without loss of generality, assume that
the random variables in the set $$\left\{\sum_{i=1}^{n}\sum_{j=1}^{n_{i}}Z_{ij}, \text{where $n_{i}\in\mathbb{N}\cup\{0\}$
and $Z_{ij}$ are i.i.d. distributed as $Z_{i}$}\right\}$$
are supported on $\mathbb{N}\cup\{0\}$.\footnote{{If these random variables are supported on $\mathbf{b}\mathbb{N}\cup\{0\}$
for some $\mathbf{b}>0$, then in $\mathbb{P}(V(t)\geq Rt)$ etc., with $Rt\in \mathbf{b}\mathbb{N}$,
it is equivalent to $\mathbb{P}(\tilde{V}(t)\geq\tilde{R}t)$ with $\tilde Rt\in\mathbb{N}$ and
$\tilde{V}(t)=V(t)/\mathbf{b}$ and $\tilde{R}=R/\mathbf{b}$.}}
Then one can readily see that $V(t)$ takes values in $\mathbb{N}\cup\{0\}$. In addition, it will be clear from Section~\ref{Sec:Infinite} that the infinite divisible distribution $\phi$ in Theorem~\ref{thm:app-mod} also takes values in $\mathbb{N}\cup\{0\}$.
 Hence to expand $\mathbb{P}(V(t)\geq Rt)$, we can assume $Rt\in\mathbb{N}$, and define
$X_{k}=V(\frac{1}{R}k)$ and $t_{k}=\frac{1}{R}k$ for $k \in \mathbb{N}$ to apply the mod-$\phi$ convergence theory.
Then, Part (1) of Theorem~\ref{thm:preciseLDP}
directly follows from Theorem~\ref{thm:app-mod}, Theorem~\ref{thm:3} and Proposition~\ref{prop:modphi-lattice}.
 Similarly, Part (2) of Theorem~\ref{thm:preciseLDP}
follows from Theorem~\ref{thm:app-mod},  Proposition~\ref{prop:modphi-nonlattice-gz} and Theorem~A in \cite{Harrington1978} which allows one to extend the convergence along any lattice (i.e. $t=b^{\ast}k$ for any $b^{\ast}>0$ and $k \in \mathbb{N}$ with $k \rightarrow \infty$
so that we can define $X_{k}=V(b^{\ast}k)$ and $t_{k}=b^{\ast}k$ to apply the mod-$\phi$ convergence theory) to  any $t \rightarrow \infty$.

We next prove Theorems~\ref{CLTThm}. It follows from Theorems 3.9 and 4.8 in \cite{modphi}, and Theorem \ref{thm:app-mod} that for $y=o(t^{1/6})$, as $t \rightarrow \infty,$
\begin{equation*}
\mathbb{P}\left(V(t)\geq t\eta'(0)+\sqrt{t\eta''(0)}y\right)
=\int_{y}^{\infty}\frac{e^{-\frac{u^{2}}{2}}}{\sqrt{2\pi}}du
\cdot(1+o(1)),
\end{equation*}
where $\eta(\cdot)$ is defined in \eqref{eq:eta-phi}. So it suffices to show $\eta'(0)=r$ and $\eta''(0)=\sigma^{2}$. From \eqref{eq:eta-phi}, we have
\begin{eqnarray*}
\eta'(0)&=&(u^{\ast})'(0)^{T}b+\sum_{i=1}^{n}\lambda_{i}\mathbb{E}\left[(1+(u^{\ast})'(0)^{T}\gamma_{i})Z_i\right],\\
\eta''(0)&=&(u^{\ast})''(0)^{T}b+\sum_{i=1}^{n}\lambda_{i}\mathbb{E}\left[((1+(u^{\ast})'(0)^{T}\gamma_{i})Z_i)^{2}
\right]
+\sum_{i=1}^{n}\lambda_{i}\mathbb{E}\left[(u^{\ast})''(0)^{T}\gamma_{i}Z_i\right].
\end{eqnarray*}
To compute $(u^{\ast})'(0)$ and $(u^{\ast})''(0)$, we can use \eqref{eq:u-ast} and find that
\begin{equation*}
\beta^{T}(u^{\ast})^{'}(0)-
\kappa^{T}\mathbb{E}[(1+(u^{\ast})^{'}(0)^{T}\gamma)Z]=0,
\end{equation*}
and
\begin{align*}
&\beta^{T}((u^{\ast})''(0) -((u^{\ast})^{'}(0))^{T}\alpha(u^{\ast})^{'}(0)
\\
&\qquad
-\kappa^{T}\mathbb{E}[((1+(u^{\ast})^{'}(0)^{T}\gamma)Z)^{2}]
-\kappa^{T}\mathbb{E}[(u^{\ast})''(0)^{T}\gamma Z]=0.
\end{align*}
Then it is straightforward to verify that $\eta'(0)=r$ and $\eta''(0)=\sigma^{2}$, where $r, \sigma^2$ are given in Theorem~\ref{thm:zhang1}. Hence, the conclusion follows.
\end{proof}

\subsection{Proof of Theorem~\ref{thm:app-mod}}\label{thm:modphi-app}
We prove Theorem~\ref{thm:app-mod} in this section. By Definition~\ref{def:mod-phi}, it suffices to show
\begin{itemize}

\item [(i)]
\begin{equation} \label{eq:converg-exp}
\sup_{\theta \in K } \left| e^{-\eta(\theta)t} \cdot \mathbb{E}\left[e^{\theta V(t)}\right] - \psi(\theta) \right| \le C_K e^{- \bar C_K t},
\end{equation}
where $K$ is a compact subset of the strip $\mathbb{S}_{(0,\theta_{c})}$ for some $\theta_{c}>0$,
\footnote{As we only study the right tail of $V(t)$,
we only need to consider $\theta \in \mathbb{S}_{(0, \theta_{c})}$. For the expression of $\theta_{c}$ and related discussions,
see Section 4 in \cite{Zhang}. }
and $C_K, \bar C_K >0$ are two constants.

\item [(ii)] For $\eta(\cdot)$ given in \eqref{eq:eta-phi}
with $\theta\in\mathbb{C}$ and $\mathcal{R}(\theta)\leq\theta_{c}$, we have $e^{\eta(\theta)}=\mathbb{E}[e^{\theta Y}]$
for some (non-constant) infinitely divisible random variable $Y$;

\end{itemize}

We present the proofs of (i) and (ii) in the next two sections.
Before we proceed,
we recall from \cite{Zhang} that they introduced $Q^*_{\theta}$,
the equivalent probability measure induced by the martingale
$M^*_{\theta}(t) = \exp[ \theta V(t) - \eta (\theta) t + u^*(\theta)^{T}(X(t) - X(0))]$, such that for given $X(0) = x_0$,
\begin{equation} \label{eq:mart}
\mathbb{E}\left[e^{\theta V(t)-\eta(\theta)t}\right]
=e^{u^{\ast}(\theta)^{T}X(0)}\mathbb{E}^{Q^*_{\theta}}\left[e^{-u^{\ast}(\theta)^{T}X(t)}\right].
\end{equation}
By Girsanov's theorem, under the measure $Q^{\ast}_{\theta}$, the process $X$ is still affine and satisfies the SDE in \eqref{eq:SDE-affine}
with parameters ($\alpha$, $a$, $b$, $\beta^{\ast}$, $\lambda^{\ast}$, $\kappa^{\ast}$, $\gamma$),
and measure $\varphi^{\ast}_{i}$, where
\begin{align}
&\lambda^{\ast}_{i}=\lambda_{i}\int_{\mathbb{R}_{+}}e^{(\theta+u^{\ast}(\theta)^{T}\gamma_{i})z}\varphi_{i}(dz),
\label{Qast} \\
&\kappa^{\ast}_{i}=\kappa_{i}\int_{\mathbb{R}_{+}}e^{(\theta+u^{\ast}(\theta)\gamma_{i})z}\varphi_{i}(dz),
\label{kappa-ast}\\
&\beta^{\ast}=
\left(
\begin{array}{cc}
\beta_{I,I}-\text{diag}(\alpha_{11}^{1}u^{\ast}_{1}(\theta),\ldots,\alpha_{mm}^{m}u_{m}^{\ast}(\theta)) & 0
\\
\beta_{J,I} & \beta_{J,J}
\end{array}
\right),
\label{beta:star}
\\
&\varphi^{\ast}_{i}(dz)=\frac{e^{(\theta+u^{\ast}(\theta)^{T}\gamma_{i})z}\varphi_{i}(dz)}
{\int_{\mathbb{R}_{+}}e^{(\theta+u^{\ast}(\theta)^{T}\gamma_{i})z}\varphi_{i}(dz)}, \label{varphi-ast}
\end{align}
and $\mathcal{R}(\theta)\leq\theta_{c}$. See \cite{Zhang}. We will use these facts in the proofs of (i) and (ii).

\subsubsection{Locally Uniform Convergence at Exponential Speed in \eqref{eq:converg-exp}}\label{Sec:Conv}

In this section, we prove \eqref{eq:converg-exp}.
First, we infer from \eqref{eq:mart} that with $X(0)=x_0$,
\begin{eqnarray}
\lim_{t \rightarrow \infty} \mathbb{E}\left[e^{\theta V(t)-\eta(\theta)t}\right]
&=&e^{u^{\ast}(\theta)^{T} x_0} \cdot \lim_{t \rightarrow \infty} \mathbb{E}^{Q^*_{\theta}}\left[e^{-u^{\ast}(\theta)^{T}X(t)}\Big| X(0)=x_0\right] \nonumber \\
&=& e^{u^{\ast}(\theta)^{T}x_{0}} \cdot \mathbb{E}^{Q^*_{\theta}}\left[e^{-u^{\ast}(\theta)^{T}X(\infty)}\Big|X(0)=x_{0}\right] \label{eq:limit1}
\end{eqnarray}
where the random vector $X(\infty)$ follows the stationary distribution of $X(\cdot)$ under the measure $Q^*_{\theta}$, and the second equality follows from the ergodicity of $X(\cdot)$, see Proposition~11 in \cite{Zhang}.

We next compute the right hand side of \eqref{eq:limit1} and show it is exactly $\psi$ given in \eqref{psiDefn}.
We can compute that $v(t,x):=\mathbb{E}^{Q^{\ast}_{\theta}}[e^{-\delta^{T}X(t)}|X(0)=x]$
satisfies the Kolmogorov equation, which is a partial integro-differential equation in our context:
\begin{align*}
\frac{\partial v}{\partial t}
&=(b-\beta^{\ast}x)\cdot\nabla v(t,x)
+\frac{1}{2}\sum_{i,j=1}^{d}\left(a_{ij}+\sum_{k=1}^{m}\alpha_{ij}^{k}x_{k}\right)\frac{\partial^{2}v}{\partial x_{i}\partial x_{j}}(x)
\\
&\qquad\qquad
+\sum_{i=1}^{n}(\lambda_{i}^{\ast}+\kappa_{i}^{\ast}\cdot x)
\int_{\mathbb{R}_{+}}[v(t,x+\gamma_{i}z)-v(t,z)]\varphi^{\ast}_{i}(dz),
\end{align*}
with the initial condition $v(0,x)=e^{-\delta^{T}x}$. Since the process $X$ is still affine under the measure $Q^{\ast}_{\theta}$, one can readily obtain that
\begin{equation} \label{eq:u-tx}
v(t,x)=e^{A(t;\theta,\delta)^{T}x+B(t;\theta,\delta)},
\end{equation}
where $A(t;\theta,\delta)$ and $B(t;\theta,\delta)$ are given in \eqref{eq:A} and \eqref{eq:B} respectively.
By ergodicity of $X(\cdot)$, we know that $\mathbb{E}^{Q^*_{\theta}}\left[e^{-u^{\ast}(\theta)^{T}X(\infty)}\Big|X(0)=x_{0}\right]$
is independent of the initial position $x_{0}$
and thus $A(\infty;\theta,u^{\ast}(\theta))=0$
and we conclude that
\begin{equation*}
e^{u^{\ast}(\theta)^{T}x_{0}} \cdot \mathbb{E}^{Q^*_{\theta}}\left[e^{-u^{\ast}(\theta)^{T}X(\infty)}\Big|X(0)=x_{0}\right]
=e^{u^{\ast}(\theta)^{T}x_{0}+B(\infty;\theta,u^{\ast}(\theta))} = \psi(\theta).
\end{equation*}

It remains to show the convergence in \eqref{eq:limit1} is exponentially fast in $t$ and locally uniformly in $\theta$. By \eqref{eq:u-tx}, we need to show $e^{A(t;\theta,u^{\ast}(\theta))^{T}x_{0}+B(t;\theta,u^{\ast}(\theta))+u^{\ast}(\theta)^{T}x_{0}}$ converges exponentially fast to $e^{B(\infty;\theta,u^{\ast}(\theta))+u^{\ast}(\theta)^{T}x_{0}}$. So it suffices to show that
\begin{equation*}
A(t;\theta,u^{\ast}(\theta))^{T} x_0 +B(t;\theta,u^{\ast}(\theta))\rightarrow B(\infty;\theta,u^{\ast}(\theta)),
\end{equation*}
exponentially fast in $t\rightarrow\infty$ locally uniformly in $\theta$.


Since $A(t;\delta)\rightarrow 0$ as $t\rightarrow\infty$,
for sufficiently large $t$, we can infer from \eqref{eq:A} that
\begin{align*}
&\frac{d}{dt}\mathcal{R}(A_{j})(t;\theta,\delta)
\leq
\sum_{i=1}^{n}\mathcal{R}(A_{i})(-\beta_{i,j}^{\ast}+\epsilon_{i,j})+\sum_{i=1}^{n}
\int_{\mathbb{R}_{+}}\mathcal{R}(A^{T})\gamma_{i}ze^{(\theta+u^{\ast}(\theta)^{T}\gamma_{i})z}\varphi_{i}(dz)\kappa_{i,j},
\\
&\frac{d}{dt}\text{Im} (A_{j})(t;\theta,\delta)
\leq
\sum_{i=1}^{n}\mathcal{I}(A_{i})(-\beta_{i,j}^{\ast}+\epsilon_{i,j})+\sum_{i=1}^{n}
\int_{\mathbb{R}_{+}}\text{Im} (A^{T})\gamma_{i}ze^{(\theta+u^{\ast}(\theta)^{T}\gamma_{i})z}\varphi_{i}(dz)\kappa_{i,j},
\end{align*}
where $\epsilon_{i,j}>0$ is sufficiently small, $\mathcal{R}(A_{j})$ and $\text{Im}(A_{j})$ take the real and imaginary parts of $A_{j}$.
Thus, for sufficiently large $t$, we have
\begin{align*}
&\frac{d\mathcal{R}(A)}{dt}\leq\left(-(\beta^{\ast})^{T}+\epsilon^{T}+(\kappa^{\ast})^{T}\mathbb{E}[\gamma Z^{\ast}]^{T}\right)\mathcal{R}(A),
\\
&\frac{d\text{Im}(A)}{dt}\leq\left(-(\beta^{\ast})^{T}+\epsilon^{T}+(\kappa^{\ast})^{T}\mathbb{E}[\gamma Z^{\ast}]^{T}\right) \text{Im}(A),
\end{align*}
where $Z^*=(Z^*_i)$ with the random variable $Z^*_i$ following the distribution $\varphi^{\ast}_{i}(dz)$, and we have used \eqref{kappa-ast}and \eqref{varphi-ast}.
Since the matrix $(\beta^{\ast})^{T}- (\kappa^{\ast})^{T}\mathbb{E}[\gamma Z^*]^{T}$ is positive stable (see the proof of Proposition~11 in \cite{Zhang}),
we deduce that
$A(t;\theta,u^{\ast}(\theta))\rightarrow 0$ exponentially fast in $t$,
which together with \eqref{eq:B} implies that $B(t;\theta,u^{\ast}(\theta))\rightarrow B(\infty;\theta,u^{\ast}(\theta))$
also exponentially fast in $t$. The convergence
holds locally uniformly in $\theta$, and thus we have proved the desired result.

\subsubsection{Infinite divisibility}\label{Sec:Infinite}

In this section, we prove that for $\eta(\cdot)$ given in \eqref{eq:eta-phi}, we have $e^{\eta (\theta)}=\mathbb{E}[e^{\theta Y}]$
for some infinitely divisible, non-constant random variable $Y$.
Recall from \eqref{eq:eta-phi} that
\begin{equation}\label{eta0}
\eta(\theta) =u^{\ast}(\theta)^{T}b
+\sum_{i=1}^{n}\lambda_{i}\left(\mathbb{E}\left[e^{(\theta+u^{\ast}(\theta)^{T}\gamma_{i})Z_i}\right]-1\right).
\end{equation}
It suffices to show that we can find two independent random variables $Y_{1}$ and $Y_{2}$
such that $Y=Y_{1}+Y_{2}$ is (non-constant) infinitely divisible and
\begin{align}
&\mathbb{E}\left[e^{\theta Y_{1}}\right]=
e^{u^{\ast}(\theta)^{T}b}, \label{eqn:Y1}
\\
&\mathbb{E}\left[e^{\theta Y_{2}}\right]=
e^{\sum_{j=1}^{n}\lambda_{j}\left(\mathbb{E}\left[e^{(\theta+u^{\ast}(\theta)^{T}\gamma_{j})Z_{j}}\right]-1\right)}.\label{eqn:Y2}
\end{align}

We first show \eqref{eqn:Y1}. Let $L_{i}(X(0),\lambda,a,b;t)$ and $V(X(0),\lambda,a,b;t)$
denote the process $L_{i}(t)$ and $V(t)$ respectively
with emphasis on its dependence on $X(0)$, $\lambda$, $a$ and $b$.
Note that by the definition of $u^{\ast}(\theta)$ in Equation \eqref{eq:u-ast},
$u^{\ast}(\theta)$ is independent of the parameters $X(0)$, $\lambda$, $a$ and $b$.
Now if $a$, $b=0$ and $\lambda=0$
\footnote{In Part (II) of Assumption \ref{assump1}, we have followed \cite{Zhang} to assume $b_{i}>0$, $1\leq i\leq m$,
so that the large deviations and central limit theorems for $V(t)$ are non-trivial. But we can indeed let $a=0$, $b=0$ and $\lambda=0$
as in part (I) of Assumption \ref{assump1} so that the parameters are \textit{admissible}, see \cite{Zhang} for details.},
then
it follows from Equation~\eqref{eta0} that $\eta(\theta)=0$.
In addition, when $a=0$, $b=0$ and $\lambda=0$,
we get from the definition of $B$ in Equation \eqref{eq:B}
and the initial condition $B(0;\theta, \delta)=0$ that $B(\infty; \theta, u^{\ast}(\theta))=0$.
Then with $\theta$ in the strip $\mathbb{S}_{(0,\theta_{c})}$, we obtain for given deterministic $X(0)$,
\begin{align}
\lim_{t\rightarrow\infty}\mathbb{E}\left[e^{\theta V(X(0),0,0,0;t)}\right]
&=\lim_{t\rightarrow\infty}\mathbb{E}\left[e^{\theta V(X(0),0,0,0;t)}\right]e^{-\eta(\theta)t}
\nonumber
\\
&=e^{u^{\ast}(\theta)^{T}X(0)+B(\infty; \theta, u^{\ast}(\theta))}
= e^{u^{\ast}(\theta)^{T}X(0)} < \infty,
\label{eq:finite}
\end{align}
where the second equality follows from the results in Section~\ref{Sec:Conv}. Moreover, since $V(\cdot,0,0,0;t)$ is monotone in $t$, we then infer from Fatou's lemma that  for real valued $\theta \in (0, \theta_c)$,
$$\mathbb{E}\left[\lim_{t\rightarrow\infty}e^{\theta V(\cdot,0,0,0;t)}\right]
\leq\lim_{t\rightarrow\infty}\mathbb{E}\left[e^{\theta V(\cdot,0,0,0;t)}\right]<\infty.$$
Hence, the random variable $V(\cdot,0,0,0;\infty) :=\lim_{t\rightarrow\infty}V(\cdot,0,0,0;t)<\infty$ with probability one, so is $L_{i}(\cdot,0,0,0;\infty):=\lim_{t\rightarrow\infty}L_{i}(\cdot,0,0,0;t)$, $1\leq i\leq n$. In addition, we obtain that for $\theta$ in the strip $\mathbb{S}_{(0,\theta_{c})}$,
\begin{align}\label{X0Eqn}
\mathbb{E}\left[e^{\theta\sum_{i=1}^{n}L_{i}(X(0),0,0,0;\infty)}\right]
&=\mathbb{E}\left[e^{\theta V(X(0),0,0,0;\infty)}\right]
\\
&=\lim_{t\rightarrow\infty}\mathbb{E}\left[e^{\theta V(X(0),0,0,0;t)}\right]
\nonumber
\\
&= e^{u^{\ast}(\theta)^{T}X(0)},
\nonumber
\end{align}
where we have used \eqref{eq:finite} and the dominated convergence theorem.
Now we define
\begin{equation*}
Y_{1}=\sum_{i=1}^{n}L_{i}(b,0,0,0;\infty),
\qquad\text{in distribution}.
\end{equation*}
Then from Equation~\eqref{X0Eqn}, we get Equation~\eqref{eqn:Y1}.

To establish \eqref{eqn:Y2}, we define
\begin{equation*}
Y_{2}=\sum_{j=1}^{n}Y_{2}^{j},\qquad\text{in distribution},
\end{equation*}
where $Y_{2}^{j}$ are independent and
\begin{equation*}
Y_{2}^{j}=\sum_{k=1}^{N_{\lambda_{j}}}({Z}_{j,k}+X_{j,k}).
\end{equation*}
Here
$Z_{j, k}$ has the same distribution as $Z_{j}$ and
$Z_{j, k}$ are i.i.d., independent of $X_{j, k}$
where $X_{j, k}$ are i.i.d. and distributed as $\sum_{i=1}^{n}L_{i}(\gamma_{j}Z_{j},0,0,0;\infty)$
and finally $N_{\lambda_{j}}$ is a Poisson random variable with mean $\lambda_{j}$
and independent of $Z_{j,k}$ and $X_{j,k}$.
Then, from Equation~\eqref{X0Eqn}, we get
\begin{align*}
\mathbb{E}\left[e^{\theta Y_{2}}\right]
&=\exp\left\{\sum_{j=1}^{n}\lambda_{j}\left(\mathbb{E}\left[e^{\theta(Z_{j}+\sum_{i=1}^{n}L_{i}(\gamma_{j} \hat Z_{j},0,0,0;\infty))}\right]-1\right)\right\}
\\
&=\exp\left\{\sum_{j=1}^{n}\lambda_{j}\left(\mathbb{E}\left[e^{(\theta+u^{\ast}(\theta)^{T}\gamma_{j})Z_{j}}\right]-1\right)\right\},
\end{align*}
where $\hat Z_j$ is an independent copy of $Z_j$.
This verifies Equation~\eqref{eqn:Y2}.

Hence, we have defined a random variable $Y=Y_{1}+Y_{2}$ so that $\mathbb{E}[e^{\theta Y}]=e^{\eta(\theta)}$.
To see that $Y$ is infinitely divisible, we can write it as i.i.d. $m$ copies of the same random variable
with parameters $b$ replaced by $b/m$ and $\lambda$ replaced by $\lambda/m$.
That is, we can define $Y=\sum_{\ell=1}^{m}Y^{\ell,(m)}$, where
$Y^{\ell,(m)}$ are i.i.d. with the same distribution as $Y_{1}^{(m)}+Y_{2}^{(m)}$,
where $Y_{1}^{(m)}$ and $Y_{2}^{(m)}$ are independent and
\begin{equation*}
Y_{1}^{(m)}=\sum_{i=1}^{n}L_{i}(b/m,0,0,0;\infty),
\qquad\text{in distribution},
\end{equation*}
and
\begin{equation*}
Y_{2}^{(m)}=\sum_{j=1}^{n}\sum_{k=1}^{N_{\lambda_{j}/m}}(Z_{j, k}+X_{j, k}),\qquad\text{in distribution},
\end{equation*}
so that
\begin{align*}
\mathbb{E}\left[e^{\theta\sum_{\ell=1}^{m}Y^{(\ell,m)}}\right]
&=\left(\mathbb{E}\left[e^{\theta\left(Y_{1}^{(m)}+Y_{2}^{(m)}\right)}\right]\right)^{m}
\\
&=\left(e^{u^{\ast}(\theta)^{T}\frac{b}{m}}e^{\sum_{j=1}^{n}\frac{\lambda_{j}}{m}\left(\mathbb{E}\left[e^{(\theta+u^{\ast}(\theta)^{T}\gamma_{j})Z_{j}}\right]-1\right)}\right)^{m}
\\
&=e^{u^{\ast}(\theta)^{T}b+\sum_{j=1}^{n}\lambda_{j}\left(\mathbb{E}\left[e^{(\theta+u^{\ast}(\theta)^{T}\gamma_{j})Z_{j}}\right]-1\right)}
=\mathbb{E}\left[e^{\theta Y}\right].
\end{align*}
This completes the proof of the infinite divisibility of $Y$.

Finally, let us show that $Y$ is lattice distributed if and only if the random variables in
the set
\begin{equation*}
\mathcal{Z}:=\left\{\sum_{i=1}^{n}\sum_{j=1}^{n_{i}}Z_{ij}, \text{where $n_{i}\in\mathbb{N}\cup\{0\}$
and $Z_{ij}$ are i.i.d. distributed as $Z_{i}$}\right\},
\end{equation*}
are supported on $\mathbf{b}\mathbb{N}\cup\{0\}$ for some $\mathbf{b}>0$.
First, we prove the `if' part. If the random variables in $\mathcal{Z}$ are supported on $\mathbf{b}\mathbb{N}\cup\{0\}$ for some $\mathbf{b}>0$,
then so is each $Z_{ij}$ and by the definition of $Y$ above, $Y$ is lattice distributed.
Next, we prove the `only if' part.
If $Y$ is lattice distributed, then $Y$ is supported on $\mathbf{a}+\mathbf{b}\mathbb{Z}$
for some $\mathbf{a}\in\mathbb{R}$ and $\mathbf{b}>0$. By the definition of $Y$ above,
we infer that each random variable in $\mathcal{Z}$
is also supported on $\mathbf{a}+\mathbf{b}\mathbb{Z}$.
Therefore, $Z_{ij}$ is also supported on $\mathbf{a}+\mathbf{b}\mathbb{Z}$.
If $Z_{ij}=\mathbf{a}+\mathbf{b}z_{ij}$
for some $z_{ij}\in\mathbb{Z}$, then $Z_{i1}+Z_{i2}=2\mathbf{a}+\mathbf{b}(z_{i1}+z_{i2})$ also belongs
to $\mathbf{a}+\mathbf{b}\mathbb{Z}$, which implies that $\mathbf{a}\in\mathbf{b}\mathbb{Z}$,
and hence we can take $\mathbf{a}=0$. Finally, if each $Z_{ij}$ is supported on $\mathbf{b}\mathbb{Z}$,
then so are the random variables in the set $\mathcal{Z}$.


\section{Numerical experiments} \label{sec:numerical}

We can develop explicit approximations for large deviation probabilities and tail expectations associated with affine point processes, by truncating the asymptotic expansions in Theorem~\ref{thm:preciseLDP}. 
In this section, we illustrate numerically the performance of our approximations
for two test cases on three-dimensional affine point processes where the model and parameters are taken from \cite{Zhang}. Recall the definition of affine point processes in Section~\ref{sec:app}. 

The model specification of the three-dimensional affine point process $(L_1, L_2, L_3)$ is as follows.
The components of $X=(X_1, X_2, X_3)$ satisfy
\begin{eqnarray*}
dX_j(t) = (b_j - \beta_j X_j(t)) dt + \sigma_j \sqrt{X_j(t)} dW_j(t) + \sum_{i=1}^{3} \gamma_i dL_i(t), \quad \text{$j=1, 2, 3$,}
\end{eqnarray*}
where $(b_1, b_2, b_3)=(6,6.1,6.2),  (\beta_1, \beta_2, \beta_3) =  (2,2.1,2.2), (\sigma_1, \sigma_2, \sigma_3)=(0.5,0.6,0.7)$, and $(\gamma_{1}, \gamma_{2}, \gamma_{3})=(0.2,0.3,0.4)$. The jump intensity is $\Lambda_j(X(t)) = \kappa_j X_j(t)$ for each $j$ with $(\kappa_1, \kappa_2, \kappa_3)=(1,1.1,1.2)$.
The two test cases we consider differ only in the distribution of marks $Z_i$. In the first test case, $Z_i$ is assumed to be constant one for each $i$. 
In the second test case, $Z_i$ is assumed to follow an exponential distribution with mean one. 

For illustrations, we focus on computing $\mathbb{P}(V(t) \ge x t)$ and $\mathbb{E}[(V(t)- xt)^+]$ for different values of $x$ and $t$ for the two test cases,
where $V(t) = \sum_{i=1}^3 L_i(t)$ and the notation $y^+ := y \cdot 1_{y \ge 0}$ for a real number $y$. The quantities we compute can be useful for applications where affine point processes
are relevant models to capture the clustering or self- and mutually-exciting effects of event arrivals, e.g., in portfolio credit risk where corporate defaults exhibit clustering. In this context, the components of the three-dimensional affine point process $(L_1, L_2, L_3)$ model the portfolio losses triggered by defaults of three firm types, and $V(t)$ represents the cumulative portfolio loss up to time $t$, see, e.g., \cite{Errais}. Hence, $\mathbb{P}(V(t) \ge x t)$ measures the probability of the cumulative portfolio loss exceeds level $xt$ and $\mathbb{E}[(V(t)- xt)^+]$ quantifies the expected exceedance of the cumulative portfolio loss to the level $xt$.

In Tables \ref{3d-lattice} and \ref{3d-exp}, we compare the results from first
order approximations and those from Monte Carlo simulations. 
The first order approximations refer to truncating the
asymptotic series in \eqref{eq:LtRt}--\eqref{eq:ES2} to retain the first two terms. In particular, for the first test case with $Z_i = 1$ for each $i$, we use \eqref{eq:LtRt} and \eqref{eq:ES}; For the second test case with $Z_i$ following exponential distributions, we use \eqref{eq:LtRt2} and \eqref{eq:ES2}.
When computing the expansion
coefficients (e.g. $\hat c_k, \hat d_k$), we need to compute the derivatives of functions $\eta$ and $\psi$ in Theorem~\ref{thm:app-mod}, see Appendix~\ref{sec:affine-coeff} for details. This requires us to numerically solve ODEs 
and we use Runge-Kutta methods.
We do not report numerical results from zero-th order approximations as the errors are not small.
For the simulations, as plain Monte Carlo is very slow in computing small probabilities and tail expectations,
we use the importance sampling algorithm in \cite{Zhang} with sample size 100,000.
The computations are performed using Python on a Mac computer with Processor 1.4 GHz Intel Core i5.

We can observe from Tables \ref{3d-lattice} and \ref{3d-exp} that the relative errors for the first order approximations compared with Monte Carlo simulations
(via importance sampling) become small when $t$ becomes large. We also experiment the
second order approximation and find similar patterns, but adding additional higher order terms does not necessarily improve the approximation as an asymptotic series is usually a non-convergent series. Nevertheless, the analytical approximations we develop are faster to evaluate than Monte Carlo simulation and they can yield accurate numerical results in the large time regime.
Hence, these approximations can be alternatives or complementary tools to simulation which is computer resource extensive.

\begin{table}[htbp]
\center
\small
    \begin{tabular}{cccccc}
           \multicolumn{1}{l}{$t$} & \multicolumn{1}{l}{$P$(IS) [95\% CI.]} & \multicolumn{1}{l}{$P$(1-Order) [R.E.]} & \multicolumn{1}{l}{$E$(IS) [95\% CI.]} & \multicolumn{1}{l}{$E$(1-Order) [R.E.]} \\

    \hline
    &&$x=25$\\
  		   10    & 2.66E-03 $[\pm 1.53\%]$& 1.47E-03 $[-44.7\%]$& 2.56E-02 $[\pm 1.31\%]$& -3.50E-02 $[-237.\%]$ \\
  		   20    & 2.64E-04 $[\pm 1.51\%]$& 2.39E-04 $[-9.47\%]$& 3.11E-03 $[\pm 1.18\%]$& 1.70E-03 $[-45.3\%]$\\
  		   30    & 2.37E-05 $[\pm 1.59\%]$& 2.31E-05 $[-2.53\%]$& 3.03E-04 $[\pm 1.19\%]$& 2.52E-04 $[-16.8\%]$\\
  		   50    & 2.03E-07 $[\pm 1.70\%]$& 2.01E-07 $[-0.99\%]$& 2.81E-06 $[\pm 1.23\%]$& 2.65E-06 $[-5.69\%]$\\
  		   100   & 1.52E-12 $[\pm 1.98\%]$& 1.53E-12 $[+0.66\%]$& 2.22E-11 $[\pm 1.39\%]$& 2.23E-11 $[+0.45\%]$ \\
  		   200   & 1.11E-22 $[\pm 2.34\%]$& 1.11E-22 $[+0.00\%]$& 1.68E-21 $[\pm 1.63\%]$& 1.68E-21 $[+0.00\%]$\\
  		   300   & 9.04E-33 $[\pm 2.54\%]$& 9.03E-33 $[-0.11\%]$& 1.39E-31 $[\pm 1.79\%]$& 1.38E-33 $[-0.72\%]$\\
    \hline
    &&$x=30$\\
		   10    & 1.39E-05 $[\pm 2.13\%]$& 1.23E-05 $[-35.3\%]$& 1.08E-04 $[\pm 1.71\%]$& 2.48E-05 $[-77.0\%]$\\
		   20    & 3.18E-08 $[\pm 2.12\%]$& 3.29E-08 $[+17.8\%]$& 2.92E-07 $[\pm 1.59\%]$& 2.70E-07 $[-7.53\%]$\\
		   30    & 6.62E-11 $[\pm 2.24\%]$& 6.84E-11 $[+3.32\%]$& 6.33E-10 $[\pm 1.62\%]$& 6.36E-10 $[+0.47\%]$\\
		   50    & 2.85E-16 $[\pm 2.43\%]$& 2.92E-16 $[+2.46\%]$& 2.86E-15 $[\pm 1.75\%]$& 2.92E-15 $[+2.10\%]$\\
		   100   & 1.23E-29 $[\pm 2.82\%]$& 1.25E-29 $[+1.62\%]$& 1.28E-28 $[\pm 1.99\%]$& 1.31E-28 $[+2.34\%]$\\
		   200   & 2.95E-56 $[\pm 3.31\%]$& 2.95E-56 $[+0.00\%]$& 3.08E-55 $[\pm 2.33\%]$& 3.14E-55 $[+1.94\%]$\\
		   300   & 7.82E-83 $[\pm 3.67\%]$& 7.88E-83 $[+0.77\%]$& 8.26E-82 $[\pm 2.57\%]$& 8.32E-82 $[+0.73\%]$\\
    \end{tabular}%
  \caption{\small{First order approximations are compared with simulations using importance sampling (IS) when $Z_i \equiv 1$ for $i=1,2 ,3$. $P$ stands for $\mathbb{P}(V(t) \ge xt)$ and $E$ stands for $\mathbb{E}[(V(t) - xt)^+]$ with $V(t) = \sum_{i=1}^3 L_i(t)$. CI stands for the confidence interval and R.E. stands for the relative errors of approximations compared with simulation results.}} \label{3d-lattice}
\end{table}

\begin{table}[htbp]
  \centering
\small
    \begin{tabular}{ccccc}
    \multicolumn{1}{l}{$t$} & \multicolumn{1}{l}{$P$(IS) [95\% CI.]} & \multicolumn{1}{l}{$P$(1-Order) [R.E.]} & \multicolumn{1}{l}{$E$(IS) [95\% CI.]} & \multicolumn{1}{l}{$E$(1-Order) [R.E.]} \\
\hline
	&&$x=25$\\
		10    & 1.99E-02 $[\pm 1.24\%]$& -9.04E-04 $[-105.\%]$& 3.51E-01 $[\pm 1.09\%]$& -1.39E+00 $[-496.\%]$\\
		20    & 5.45E-03 $[\pm 1.26\%]$& 3.75E-03 $[-31.2\%]$& 1.19E-01 $[\pm 1.00\%]$& -2.63E-02 $[-122.\%]$\\
		30    & 1.47E-03 $[\pm 1.32\%]$& 1.23E-03 $[-16.3\%]$& 3.50E-02 $[\pm 1.00\%]$& 1.50E-02 $[-57.1\%]$\\
		50    & 1.11E-04 $[\pm 1.42\%]$& 1.02E-04 $[-8.11\%]$& 2.87E-03 $[\pm 1.04\%]$& 2.20E-03 $[-23.3\%]$\\
		100   & 1.96E-07 $[\pm 1.64\%]$& 1.88E-07 $[-4.08\%]$& 5.48E-06 $[\pm 1.16\%]$& 5.05E-06 $[-7.85\%]$\\
		200   & 7.61E-13 $[\pm 1.95\%]$& 7.55E-13 $[-0.79\%]$& 2.24E-11 $[\pm 1.35\%]$& 2.19E-11 $[-2.23\%]$\\
		300   & 3.39E-18 $[\pm 2.13\%]$& 3.37E-18 $[-0.59\%]$& 1.02E-16 $[\pm 1.49\%]$& 1.00E-16 $[-1.96\%]$\\
	\hline
	&&$x=30$\\
		10    & 1.10E-03 $[\pm 1.61\%]$& 6.60E-04 $[-40.0\%]$& 1.71E-02 $[\pm 1.29\%]$& -8.75E-03 $[-151.\%]$\\
		20    & 4.29E-05 $[\pm 1.66\%]$& 3.74E-05 $[-12.8\%]$& 7.67E-04 $[\pm 1.24\%]$& 4.81E-04 $[-37.3\%]$\\
		30    & 1.57E-06 $[\pm 1.76\%]$& 1.47E-06 $[-6.37\%]$& 3.01E-05 $[\pm 1.28\%]$& 2.46E-05 $[-18.3\%]$\\
		50    & 2.23E-09 $[\pm 1.95\%]$& 2.17E-09 $[-2.69\%]$& 4.46E-08 $[\pm 1.39\%]$& 4.14E-08 $[-7.17\%]$\\
		100   & 2.03E-16 $[\pm 2.26\%]$& 2.01E-16 $[-0.99\%]$& 4.24E-15 $[\pm 1.60\%]$& 4.14E-15 $[-2.36\%]$\\
		200   & 2.14E-30 $[\pm 2.70\%]$& 2.17E-30 $[+1.40\%]$& 4.57E-29 $[\pm 1.88\%]$& 4.63E-29 $[+1.31\%]$\\
		300   & 2.64E-44 $[\pm 2.97\%]$& 2.66E-44 $[+0.76\%]$& 5.70E-43 $[\pm 2.08\%]$& 5.72E-43 $[+0.35\%]$\\
    \end{tabular}%

\caption{\small{First order approximations are compared with simulations using importance sampling (IS) when $Z_i$ is exponentially distributed with mean one for $i=1,2 ,3$. $P$ stands for $\mathbb{P}(V(t) \ge xt)$ and $E$ stands for $\mathbb{E}[(V(t) - xt)^+]$ with $V(t) = \sum_{i=1}^3 L_i(t)$. CI stands for the confidence interval and R.E. stands for the relative errors of approximations compared with simulation results.}} \label{3d-exp}
\end{table}

\section*{Acknowledgements}
We thank Jim Dai, Jayaram Sethuraman, S. R. S. Varadhan, Xunyu Zhou for helpful comments,
and Shujian Liao, Hongyi Jiang for their assistance on numerical experiments.
Xuefeng Gao acknowledges support from Hong Kong RGC Grants 24207015 and 14201117.
Lingjiong Zhu is grateful to the support from NSF Grant DMS-1613164.



\newpage
\appendix

\section{Proofs of Propositions~\ref{prop:modphi-lattice} and \ref{prop:modphi-nonlattice-gz}} \label{appendixA}
This Appendix collects the proofs of Propositions~\ref{prop:modphi-lattice} and \ref{prop:modphi-nonlattice-gz}.
\subsection{Proof of Proposition~\ref{prop:modphi-lattice}}
\begin{proof}
We follow the proof of Theorem 3.4. (Theorem~\ref{thm:3} in our paper) and Remark 3.7 in \cite{modphi}.
Before we proceed, since we assume that $|g(x)|\leq\bar{a}e^{\bar hx}$ for some $\bar h<h$ and $\bar{a}>0$,
we have
\begin{equation*}
\mathbb{E}\left[g(X_{n}-t_{n}x)1_{X_{n}\geq t_{n}x}\right]
\leq\bar{a}\mathbb{E}\left[e^{\bar{h} X_{n}}\right]<\infty,
\end{equation*}
since $\bar{h}<h\leq d$, where we recall that $\mathbb{E}[e^{zX_{n}}]$
exists on $\mathbb{S}_{(c,d)}$.
Moreover, since we assume that $|g(x)|\leq\bar{a}e^{\bar hx}$ for some $\bar h<h$ and $\bar{a}>0$,
we can check that $\hat{c}_{k}$ whose definition
involves the summation of $g(q)e^{-qh}(-q)^{m}$ over $q$ from $0$ to $\infty$ is well-defined.

Next, let us define:
\begin{align}
H_{n}(w)&:=\sum_{k=0}^{\infty}g(k)e^{-kh}e^{-k\frac{iw}{\sqrt{t_{n}\eta''(h)}}}
\cdot\psi\left(h+\frac{iw}{\sqrt{t_{n}\eta''(h)}}\right)
\nonumber \\
&\qquad\qquad\qquad\qquad\cdot
e^{t_{n}\left(\eta\left(h+\frac{iw}{\sqrt{t_{n}\eta''(h)}}\right)
-\eta(h)-\eta'(h)\frac{iw}{\sqrt{t_{n}\eta''(h)}}+\frac{w^{2}}{2t_{n}}\right)}
\nonumber \\
&=\sum_{k=0}^{\infty}g(k)e^{-kh}e^{-k\frac{iw}{\sqrt{t_{n}\eta''(h)}}}
\cdot\psi\left(h+\frac{iw}{\sqrt{t_{n}\eta''(h)}}\right)
\nonumber \\
&\qquad\qquad\qquad\qquad\cdot
e^{\frac{-w^{2}}{\eta''(h)}\sum_{\ell=1}^{\infty}\frac{\eta^{(\ell+2)}(h)}{(\ell+2)!}
\left(\frac{iw}{\sqrt{t_{n}\eta''(h)}}\right)^{\ell}}.
\label{eq:gn-formula}
\end{align}
Since we assume that $|g(x)|\leq\bar{a}e^{\bar hx}$ for some $\bar h<h$ and $\bar{a}>0$, $H_{n}(w)$
is well-defined. By Taylor expansion,
we define $\beta_{j}(w)$ via the expansion:
\begin{equation*}
H_{n}(w)=\sum_{k=0}^{2v-1}\frac{\beta_{k}(w)}{(t_{n})^{k/2}}
+O(t_{n}^{-v}).
\end{equation*}
Following the proof of Theorem 3.4. and Remark 3.7. in \cite{modphi},
we can readily obtain
\begin{equation*}
\mathbb{E}\left[g(X_{n}-t_{n}x)1_{X_{n}\geq t_{n}x}\right]
=\frac{e^{-t_{n}F(x)}}{\sqrt{2\pi t_{n}\eta''(h)}}   \left(\hat c_0 + \frac{\hat c_1}{t_n} + \ldots + \frac{\hat c_{v-1}}{t_n^{v-1}} + O\left(\frac{1}{t_n^v}\right) \right),
\end{equation*}
where
\begin{equation} \label{eq:hatck-beta}
\hat{c}_{k}=\int_{\mathbb{R}}\beta_{2k}(w)\frac{e^{-\frac{w^{2}}{2}}}{\sqrt{2\pi}}dw.
\end{equation}

We next compute the coefficients $\hat{c}_{k}$. As $\beta_{j}(w)$ is defined via the expansion of $g_n(w)$, we proceed to expand the three terms in \eqref{eq:gn-formula}.

(1) First, we expand the term $\sum_{k=0}^{\infty}g(k)e^{-kh}e^{-k\frac{iw}{\sqrt{t_{n}\eta''(h)}}}$.
We can compute that
\begin{align*}
\sum_{k=0}^{\infty}g(k)e^{-kh}e^{-k\frac{iw}{\sqrt{t_{n}\eta''(h)}}}
&=\sum_{k=0}^{\infty}g(k)e^{-kh}\sum_{m=0}^{\infty}\frac{(-k)^{m}}{m!}
\left(\frac{iw}{\sqrt{t_{n}\eta''(h)}}\right)^{m}
\\
&=\sum_{m=0}^{\infty}\left[\sum_{q=0}^{\infty}g(q)e^{-qh}\frac{(-q)^{m}}{m!}\right]
\left(\frac{iw}{\sqrt{t_{n}\eta''(h)}}\right)^{m}.
\end{align*}

(2) Next, as $\psi$ is analytic, we expand the term $\psi\left(h+\frac{iw}{\sqrt{t_{n}\eta''(h)}}\right)$ as
\begin{equation*}
\psi\left(h+\frac{iw}{\sqrt{t_{n}\eta''(h)}}\right)
=\sum_{\ell=0}^{\infty}\frac{\psi^{(\ell)}(h)}{\ell!}
\left(\frac{iw}{\sqrt{t_{n}\eta''(h)}}\right)^{\ell}.
\end{equation*}

(3) Finally, let us expand the term $e^{\frac{-w^{2}}{\eta''(h)}\sum_{\ell=1}^{\infty}\frac{\eta^{(\ell+2)}(h)}{(\ell+2)!}
\left(\frac{iw}{\sqrt{t_{n}\eta''(h)}}\right)^{\ell}}$.
Let us define $f(x)=e^{-x}$ and
\begin{equation*}
\tilde f(x)=\frac{w^{2}}{\eta''(h)}\sum_{k=1}^{\infty}\frac{\eta^{(k+2)}(h)}{(k+2)!}x^{k}.
\end{equation*}
Then
\begin{equation*}
f\left(\tilde f(x)\right)=\sum_{n=0}^{\infty}\frac{d^{n}}{dx^{n}}f\left(\tilde f (0)\right)\frac{x^{n}}{n!},
\end{equation*}
where we can compute that  $\tilde f(0)=0$ and
for every $j\in\mathbb{N}$,
\begin{equation*}
\tilde f^{(j)}(0)=\frac{w^{2}}{\eta''(h)}\frac{\eta^{(j+2)}(h)}{(j+2)(j+1)},
\end{equation*}
which implies that by Fa\`{a} di Bruno's formula (see Appendix~\ref{sec:Faa})
\begin{equation} \label{eq:fadi}
f\left(\tilde f(x)\right)=\sum_{n=0}^{\infty}
\sum_{\mathcal{S}_{n}}\frac{n!(-1)^{m_{1}+\cdots+m_{n}}}{m_{1}!1!^{m_{1}}m_{2}!2!^{m_{2}}\cdots m_{n}!n!^{m_{n}}}
\cdot
\prod_{j=1}^{n}\left(\frac{w^{2}}{\eta''(h)}\frac{\eta^{(j+2)}(h)}{(j+2)(j+1)}\right)^{m_{j}}
\frac{x^{n}}{n!}.
\end{equation}
Hence,
\begin{align*}
&e^{\frac{-w^{2}}{\eta''(h)}\sum_{\ell=1}^{\infty}\frac{\eta^{(\ell+2)}(h)}{(\ell+2)!}
\left(\frac{iw}{\sqrt{t_{n}\eta''(h)}}\right)^{\ell}}
\\
&=\sum_{n=0}^{\infty}
\sum_{\mathcal{S}_{n}}\frac{(-1)^{m_{1}+\cdots+m_{n}}}{m_{1}!1!^{m_{1}}m_{2}!2!^{m_{2}}\cdots m_{n}!n!^{m_{n}}}
\cdot
\prod_{j=1}^{n}\left(\frac{w^{2}}{\eta''(h)}\frac{\eta^{(j+2)}(h)}{(j+2)(j+1)}\right)^{m_{j}}
\left(\frac{iw}{\sqrt{t_{n}\eta''(h)}}\right)^{n}.
\end{align*}

Hence, by (1), (2) and (3), we conclude that
\begin{align*}
\beta_{k}(w)
&=\sum_{m+\ell+n=k}
\sum_{q=0}^{\infty}g(q)e^{-qh}\frac{(-q)^{m}}{m!}
\cdot
\frac{\psi^{(\ell)}(h)}{\ell!}
\\
&\qquad
\cdot
\sum_{\mathcal{S}_{n}}\frac{(-1)^{m_{1}+\cdots+m_{n}}}{m_{1}!1!^{m_{1}}m_{2}!2!^{m_{2}}\cdots m_{n}!n!^{m_{n}}}
\\
&\qquad\qquad\qquad\qquad
\cdot
\prod_{j=1}^{n}\left(\frac{1}{\eta''(h)}\frac{\eta^{(j+2)}(h)}{(j+2)(j+1)}\right)^{m_{j}} \cdot
\frac{(i)^{k}w^{k+2(m_{1}+\cdots+m_{n})}}{(\eta''(h))^{k/2}}.
\end{align*}
By the property of standard normal random variable, we have
$\int_{-\infty}^{\infty}w^{m}\cdot\frac{e^{-\frac{w^{2}}{2}}}{\sqrt{2\pi}}dw
=(m-1)!!$
if $m$ is even, and $0$ if $m$ is odd.
Hence, we obtain from \eqref{eq:hatck-beta} the formula of $\hat c_k$ given in \eqref{eq:hatck}. The proof is therefore complete.
\end{proof}

\subsection{Proof of Proposition~\ref{prop:modphi-nonlattice-gz}}

To prove Proposition~\ref{prop:modphi-nonlattice-gz}, it suffices to prove Part (ii).
Before we proceed to the proof,
let us first show that $\mathbb{E}[g(X_{n}-t_{n}x)1_{X_{n}\geq t_{n}x}]$ and $\hat{d}_{k}$
are well-defined and finite.
Since we assume that $g(x)=\sum_{k=0}^{\infty}g_{k}x^{k+\Delta}$,
where $\Delta\in[0,1)$ and $g_{k}\leq\bar{a}\frac{\bar{h}^{k}}{k!}$
for some $\bar{a}>0$ and $0<\bar{h}<h$ for every $k\in\mathbb{N}$,
we infer that $g(x)\leq\bar{a}e^{\bar{h}x}x^{\Delta}$ for any $x\geq 0$,
and hence there exist some $\bar{a}_{0}>0$
and $\bar{h}<\bar{h}_{0}<h$ such that $g(x)\leq\bar{a}_{0}e^{\bar{h}_{0}x}$ for any $x\geq 0$.
Therefore
\begin{equation*}
\mathbb{E}\left[g(X_{n}-t_{n}x)1_{X_{n}\geq t_{n}x}\right]
\leq\bar{a}_{0}\mathbb{E}\left[e^{\bar{h}_{0}X_{n}}\right]<\infty,
\end{equation*}
since $\bar{h}_{0}<h\leq d$, where we recall that $\mathbb{E}[e^{zX_{n}}]$
exists on $\mathbb{S}_{(c,d)}$.
Moreover, we assume that $g(x)=\sum_{k=0}^{\infty}g_{k}x^{k+\Delta}$,
where $\Delta\in[0,1)$ and $g_{k}\leq\bar{a}\frac{\bar{h}^{k}}{k!}$
for some $\bar{a}>0$ and $0<\bar{h}<h$ for every $k\in\mathbb{N}$,
we can check that $\hat{d}_{k}$ whose definition
involves the summation of $g_{q}\frac{\Gamma(k+\Delta+q+(m_{1}+\cdots+m_{\ell})-m+1)}{h^{\Delta+q+1}}$
over $q$ from $0$ to $\infty$ is well-defined.

We next introduce a lemma. It calculates asymptotic expansions for certain Gaussian integrals which will be used in our computations in proving
Proposition~\ref{prop:modphi-nonlattice-gz}. The proof of this lemma will be deferred to the end of this section.

\begin{lemma}\label{yintegral2}
Fix any $s\geq 0$. We have as $t_{n}\rightarrow\infty$,
\begin{equation*}
\int_{y=0}^{\infty}y^{s}e^{-\frac{(y+h\sqrt{t_{n}\eta''(h)})^{2}}{2}}dy
\sim e^{-h^{2}t_{n}\eta''(h)\frac{1}{2}}
\sum_{k=0}^{\infty}\frac{\Gamma(k+s+1)a_{k}}{(h^{2}t_{n}\eta''(h))^{\frac{s}{2}+\frac{1}{2}+k}},
\end{equation*}
where $(a_{k})_{k=0}^{\infty}$ are determined recursively as:
\begin{equation*}
a_{k}=(-1)^{k}-\sum_{j=0}^{k-1}\frac{(j+s)_{k-j}}{(k-j)!2^{k-j}}a_{j},
\end{equation*}
where $(\cdot)_{i}$ is the Pochhammer symbol.
\end{lemma}

We are now ready to prove Proposition~\ref{prop:modphi-nonlattice-gz} below.

By the expansion $g(z)=\sum_{k=0}^{\infty}g_{k}z^{\Delta+k}$, it suffices
to show our results for $g(z)=z^{\gamma}$ for every $\gamma\geq 0$. Fix $h$ where $\eta'(h)=x$.
As in Lemma~4.7  of \cite{modphi}, we
denote $\tilde X_n$ as the random variable which follows the law $Q(dy) = \frac{e^{hy}}{\phi_{X_n}(h)} \mathbb{P}(X_n \in dy)$.
Denote $F_n(\cdot)$ the cumulative distribution function of the random variable
\begin{equation*}
\frac{\tilde X_n - t_n \eta'(h)}{\sqrt{t_n \eta''(h)}}.
\end{equation*}
Then as in the proof of Theorem~4.3 in \cite{modphi}, one can readily compute that
\begin{align}
&\mathbb{E}\left[(X_{n}-t_{n}x)^{\gamma}1_{X_{n}\geq t_{n}x}\right]
\nonumber \\
&=\varphi_{X_{n}}(h)\int_{y=t_n x}^{\infty} (y- t_n x)^{\gamma} e^{-hy} Q(dy)
\nonumber \\
&=\varphi_{X_{n}}(h)\int_{u=0}^{\infty}\left(\sqrt{t_{n}\eta''(h)}u\right)^{\gamma}e^{-h(t_{n}\eta'(h)+\sqrt{t_{n}\eta''(h)}u)}dF_{n}(u)
\nonumber \\
&=\psi_{n}(h)e^{-t_{n}F(x)}\left(\sqrt{t_{n}\eta''(h)}\right)^{\gamma}\int_{u=0}^{\infty}u^{\gamma}e^{-h\sqrt{t_{n}\eta''(h)}u}dF_{n}(u)
\nonumber \\
&=\psi_{n}(h)e^{-t_{n}F(x)}\left(\sqrt{t_{n}\eta''(h)}\right)^{\gamma}
\nonumber \\
&\qquad\qquad\cdot
\int_{u=0}^{\infty}\left[h\sqrt{t_{n}\eta''(h)}u^{\gamma}-\gamma u^{\gamma-1}\right]e^{-h\sqrt{t_{n}\eta''(h)}u}(F_{n}(u)-F_{n}(0))du, \label{eq:int1}
\end{align}
where we used integration by parts in the last equality.

Suppose we can find $G_{n}(x)=\int_{-\infty}^{x}g_{n}(y)dy$,
such that
\begin{equation}
\sup_{x\in\mathbb{R}}|F_{n}(x)-G_{n}(x)|=o(t_{n}^{-v}), \label{eq:F-G}
\end{equation}
where
\begin{equation}
g_{n}(y)=
\frac{1}{\sqrt{2\pi}}e^{-y^{2}/2}
\left[1+ \sum_{j=1}^{2v}\frac{Q_{j}(y)}{t_{n}^{j/2}}
\right],  \label{eq:gn}
\end{equation}
with $Q_{j}(y)$ being polynomials of order $j$ in $y$.
Then we obtain from \eqref{eq:int1} that
\begin{align} \label{eq: gamma-compute}
&\mathbb{E}\left[(X_{n}-t_{n}x)^{\gamma}1_{X_{n}\geq t_{n}x}\right]
\nonumber  \\
&=\psi_{n}(h)e^{-t_{n}F(x)}\left(\sqrt{t_{n}\eta''(h)}\right)^{\gamma}
\nonumber  \\
&\qquad\qquad\cdot
\int_{u=0}^{\infty}\left[h\sqrt{t_{n}\eta''(h)}u^{\gamma}-\gamma u^{\gamma-1}\right]e^{-h\sqrt{t_{n}\eta''(h)}u}(G_{n}(u)-G_{n}(0)
+o(t_{n}^{-v}))du
\nonumber  \\
&=\psi_{n}(h)e^{-t_{n}F(x)}\left(\sqrt{t_{n}\eta''(h)}\right)^{\gamma}
\nonumber  \\
&\qquad\qquad\cdot
\int_{u=0}^{\infty}\bigg[\int_{y=0}^{u}\left[h\sqrt{t_{n}\eta''(h)}u^{\gamma}-\gamma u^{\gamma-1}\right]e^{-h\sqrt{t_{n}\eta''(h)}u}
\cdot g_{n}(y)dy\bigg]du
\nonumber \\
&\qquad
+\psi_{n}(h)e^{-t_{n}F(x)}\left(\sqrt{t_{n}\eta''(h)}\right)^{\gamma}
\int_{u=0}^{\infty}\left[h\sqrt{t_{n}\eta''(h)}u^{\gamma}-\gamma u^{\gamma-1}\right]e^{-h\sqrt{t_{n}\eta''(h)}u}du
\cdot
o\left(t_{n}^{-v}\right)
\nonumber  \\
&=\psi_{n}(h)e^{-t_{n}F(x)}\left(\sqrt{t_{n}\eta''(h)}\right)^{\gamma} \cdot \left(
\int_{y=0}^{\infty}y^{\gamma}e^{-h\sqrt{t_{n}\eta''(h)}y}
g_{n}(y)dy
+1_{\gamma=0}\cdot o(t_{n}^{-v})  \right),
\end{align}
where we used
\begin{equation*}
\int_{u=0}^{\infty}\left[h\sqrt{t_{n}\eta''(h)}u^{\gamma}-\gamma u^{\gamma-1}\right]e^{-h\sqrt{t_{n}\eta''(h)}u}du
=y^{\gamma}e^{-h\sqrt{t_{n}\eta''(h)}y}
\bigg|_{y=\infty}^{0}=1_{\gamma=0}.
\end{equation*}
To calculate and approximate the integral in \eqref{eq: gamma-compute}, in view of the expression of $g_n(\cdot)$ in \eqref{eq:gn}, it suffices to note that
\begin{equation*}
\int_{y=0}^{\infty}y^{\gamma}e^{-h\sqrt{t_{n}\eta''(h)}y}
\frac{1}{\sqrt{2\pi}}e^{-\frac{y^{2}}{2}}y^{j}dy
=e^{h^{2}t_{n}\eta''(h)\frac{1}{2}}
\frac{1}{\sqrt{2\pi}}\int_{y=0}^{\infty}y^{\gamma+j}
e^{-\frac{(y+h\sqrt{t_{n}\eta''(h)})^{2}}{2}}dy,
\end{equation*}
and by Lemma \ref{yintegral2}, for any $s\geq 0$,
\begin{equation*}
\int_{y=0}^{\infty}y^{s}e^{-\frac{(y+h\sqrt{t_{n}\eta''(h)})^{2}}{2}}dy
\sim e^{-h^{2}t_{n}\eta''(h)\frac{1}{2}}
\sum_{k=0}^{\infty}\frac{\Gamma(k+s+1)a_{k}}{(h^{2}t_{n}\eta''(h))^{\frac{s}{2}+\frac{1}{2}+k}}
=e^{-h^{2}t_{n}\eta''(h)\frac{1}{2}}O\left(t_{n}^{-\frac{s}{2}-\frac{1}{2}}\right),
\end{equation*}
as this implies that
\begin{equation} \label{eq:int3}
\int_{0}^{\infty}y^{\gamma} e^{-h\sqrt{t_{n}\eta''(h)}y} \frac{1}{\sqrt{2\pi}}e^{-y^{2}/2}
\frac{Q_{j}(y)}{t_{n}^{j/2}}dy
=O\left(t_{n}^{-\frac{\gamma+j}{2}-\frac{1}{2}-\frac{j}{2}}\right)
=O\left(t_{n}^{-\frac{\gamma}{2}-\frac{1}{2}-j}\right).
\end{equation}
It then follows from \eqref{eq: gamma-compute} and the fact that
 $\psi_{n}(h)$ converges to $\psi(h)$ with speed $O(t_{n}^{-v})$ locally uniformly that
for $g(x) = x^{\gamma}$, we have
the expansion of the form
\begin{equation*}
\mathbb{E}\left[g(X_{n}-t_{n}x)1_{X_{n}\geq t_{n}x}\right]
=\frac{e^{-t_{n}F(x)}}{\sqrt{2\pi}\sqrt{t_{n}\eta''(h)}}
\left[\hat{d}_{0}+\frac{\hat{d}_{1}}{t_{n}}+\frac{\hat{d}_{2}}{t_{n}^{2}}+\cdots
+\frac{\hat{d}_{n-1}}{t_{n}^{v-1}}+o\left(t_{n}^{-v+1}\right)\right].
\end{equation*}
To complete the proof, it remains to find the function $G_n(\cdot)$ so that \eqref{eq:F-G} holds and identify the coefficients
$\hat{d}_{k}$ in the above expansion. This will be done in the next section.



\subsubsection{Finding $G_n(\cdot)$ and computing $\hat{d}_{k}$}\label{sec:hatdk}

In this section we show how to find the function $G_n(\cdot)$ (or equivalently $g_n(\cdot)$)
so that \eqref{eq:F-G} holds, and we also compute the expansion coefficients $\hat{d}_{k}$.

To find $G_n(\cdot)$ which approximates the distribution function $F_n(\cdot)$, we approximate the Fourier transform of the distribution $F_n(\cdot)$, and then take the inverse Fourier transform to obtain $g_n(\cdot)$. To this end, we write the Fourier transform $f_n^{\ast}(\zeta)= \int_{\mathbb{R}} e^{i \zeta x} d F_n(x)$, and
\begin{equation*}
\tilde{\psi}(z)=\frac{\psi(z+h)}{\psi(h)},
\qquad
\text{and}
\qquad
\tilde{\eta}(z)=\eta(z+h)-\eta(h).
\end{equation*}
Then for every $k\in\mathbb{N}$,
\begin{equation*}
\tilde{\psi}^{(k)}(0)=\frac{\psi^{(k)}(h)}{\psi(h)},
\qquad\text{and}
\qquad
\tilde{\eta}^{(k)}(0)=\eta^{(k)}(h).
\end{equation*}
We can compute that
for $z= i \zeta$,
\begin{equation} \label{eq:transformF}
f_n^*(\zeta):={\mathbb{E}}\left[e^{z\frac{\tilde X_{n}-t_{n} {\eta}'(h)}{\sqrt{t_{n} {\eta}''(h)}}}\right]
=e^{t_{n}\left(\tilde{\eta}\left(\frac{z}{\sqrt{t_{n} {\eta}''(h)}}\right)- {\eta}'(h)\frac{z}{\sqrt{t_{n} {\eta}''(h)}}\right)}
\cdot \tilde{\psi}_{n}\left(\frac{z}{\sqrt{t_{n} {\eta}''(h)}}\right),
\end{equation}
where
\begin{equation*}
\tilde{\psi}_{n}(z) :=e^{-t_n \tilde \eta(z)} \cdot \mathbb{E}\left[e^{z \tilde X_{n}}\right].
\end{equation*}
By Lemma~4.7 in \cite{modphi}, $\tilde{\psi}_{n}$ converges to $\tilde{\psi}$ locally uniformly with speed $O(t_n^{-v})$,
{and together with Taylor expansion, one obtains
\begin{equation*}
\tilde{\psi}_{n} \left(\frac{z}{\sqrt{t_{n} {\eta}''(h)}}\right) = \sum_{i=0}^{2v} \frac{\tilde{\psi}^{(i)}(0)}{i!} \left(\frac{z}{\sqrt{t_{n} {\eta}''(h)}}\right)^i
+ o\left((z/\sqrt{t_n})^{2v}\right).
\end{equation*}}
In addition, as in \eqref{eq:fadi}, one can expand the following term in \eqref{eq:transformF}
$$e^{t_{n}\left(\tilde{\eta}\left(\frac{z}{\sqrt{t_{n} {\eta}''(h)}}\right)- {\eta}'(h)\frac{z}{\sqrt{t_{n} {\eta}''(h)}}\right)} = e^{\frac{z^{2}}{{\eta}''(h)}\sum_{k=1}^{\infty}\frac{ {\eta}^{(k+2)}(h)}{(k+2)!}
\left(\frac{z}{\sqrt{t_{n}{\eta}''(h)}}\right)^{k}}$$ by applying the Fa\`{a} di Bruno's formula (see Section \ref{sec:Faa}). This leads to the following: 
\begin{equation} \label{eq:fn-gn}
f_{n}^{\ast}(\zeta)=\mathbf{g}_{n}^{\ast}(\zeta)
+e^{\frac{z^{2}}{2}}P_{2v+1}(|z|)\left(\frac{|z|}{\sqrt{t_{n}}}\right)^{2v}\varepsilon\left(\frac{z}{\sqrt{t_{n}}}\right),
\end{equation}
where $\lim_{t\rightarrow 0}\varepsilon(t)=0$, and $P_{i}(|z|)$ is some polynomial of $|z|$ of order $i$, and for $z= i \zeta$,
\begin{align}\label{eq:approx-F}
&\mathbf{g}_n^{\ast}(\zeta)=\sum_{k=0}^{2v}e^{\frac{z^{2}}{2}}
\sum_{\ell=0}^{k}\frac{\psi^{(k-\ell)}(h)}{\psi(h)(k-\ell)!}
\sum_{\mathcal{S}_{\ell}}\frac{1}{m_{1}!1!^{m_{1}}m_{2}!2!^{m_{2}}\cdots m_{\ell}!\ell!^{m_{\ell}}}
\nonumber \\
&\qquad\qquad\qquad
\cdot\prod_{j=1}^{\ell}\left(\frac{1}{\eta''(h)}\frac{\eta^{(j+2)}(h)}{(j+2)(j+1)}\right)^{m_{j}}
\frac{z^{k+2(m_{1}+\cdots+m_{\ell})}}{(\eta''(h))^{k/2}}\frac{1}{t_{n}^{k/2}}.
\end{align}
It is known that
\begin{equation}\label{eq:expand1}
\frac{1}{2\pi}\int_{-\infty}^{\infty}e^{-ity}e^{-t^{2}/2}(it)^{k}dt
=H_{k}(y)\frac{1}{\sqrt{2\pi}}e^{-\frac{y^{2}}{2}},
\end{equation}
where $H_{k}(y)$ are Hermite polynomials given by
\begin{equation} \label{eq:hermite}
H_{k}(y)=k!\sum_{m=0}^{\lfloor k/2\rfloor}\frac{(-1)^{m}}{m!(k-2m)!}\frac{y^{k-2m}}{2^{m}}.
\end{equation}
Therefore, by \eqref{eq:approx-F},  \eqref{eq:expand1} and $z=i\zeta$,
the inverse Fourier transform of $\mathbf{g}_n^{\ast}(\cdot)$
is given by
\begin{align}\label{fny}
g_{n}(y)
&:=\frac{1}{2\pi}\int_{-\infty}^{\infty}\mathbf{g}_{n}^{*}(\zeta)e^{-iy\zeta}d\zeta
\nonumber
\\
&=\frac{1}{\sqrt{2\pi}}e^{-\frac{y^{2}}{2}}
\sum_{k=0}^{2v}
\sum_{\ell=0}^{k}\frac{\psi^{(k-\ell)}(h)}{\psi(h)(k-\ell)!}
\sum_{\mathcal{S}_{\ell}}\frac{1}{m_{1}!1!^{m_{1}}m_{2}!2!^{m_{2}}\cdots m_{\ell}!\ell!^{m_{\ell}}}
\nonumber
\\
&\qquad\qquad\qquad
\cdot\prod_{j=1}^{\ell}\left(\frac{1}{\eta''(h)}\frac{\eta^{(j+2)}(h)}{(j+2)(j+1)}\right)^{m_{j}}
\frac{H_{k+2(m_{1}+\cdots+m_{\ell})}(y)}{(\eta''(h))^{k/2}}\frac{1}{t_{n}^{k/2}}.
\end{align}
We will show later that for $G_n(x) =\int_{-\infty}^x g_n(y)dy$ with $g_n(\cdot)$ given in \eqref{fny}, we have
\begin{equation*}
\sup_{x\in\mathbb{R}}|F_{n}(x)-G_{n}(x)|=o(t_{n}^{-v}).
\end{equation*}
By using the expansion for Hermite polynomials in \eqref{eq:hermite}, one can readily express $g_n(\cdot)$ in the form of \eqref{eq:gn}. Then similar arguments in \eqref{eq: gamma-compute}-\eqref{eq:int3} yield that
\begin{align*}
&\mathbb{E}\left[(X_{n}-t_{n}x)^{\gamma}1_{X_{n}\geq t_{n}x}\right]
\\
&=e^{-t_{n}F(x)}
\frac{1}{\sqrt{2\pi t_{n}\eta''(h)}}\frac{1}{h^{\gamma+1}}
\\
&\qquad\cdot\sum_{k=0}^{2v}
\sum_{\ell=0}^{k}\frac{\psi^{(k-\ell)}(h)}{(k-\ell)!}
\sum_{\mathcal{S}_{\ell}}\frac{1}{m_{1}!1!^{m_{1}}m_{2}!2!^{m_{2}}\cdots m_{\ell}!\ell!^{m_{\ell}}}
\\
&\qquad\qquad
\cdot\prod_{j=1}^{\ell}\left(\frac{1}{\eta''(h)}\frac{\eta^{(j+2)}(h)}{(j+2)(j+1)}\right)^{m_{j}}
\frac{1}{(\eta''(h))^{k/2}}
\\
&\qquad\cdot
\sum_{m=0}^{\lfloor\frac{k}{2}+m_{1}+\cdots+m_{\ell}\rfloor}
\frac{2^{m}(-1)^{m}(k+2(m_{1}+\cdots+m_{\ell}))!}{m!(k+2(m_{1}+\cdots+m_{\ell})-2m)!}
\\
&\qquad\cdot
\sum_{k+(m_{1}+\cdots+m_{\ell})-m+q\leq v}\frac{\Gamma(q+\gamma+k+2(m_{1}+\cdots+m_{\ell})-2m+1)a_{q}}{(h^{2}\eta''(h))^{\frac{k}{2}
+(m_{1}+\cdots+m_{\ell})-m+q} \cdot t_{n}^{k+(m_{1}+\cdots+m_{\ell})-m+q}}
+o(t_n^{-v+1}).
\end{align*}
By comparing the above equation with the expansion
\begin{equation*}
\mathbb{E}\left[(X_{n}-t_{n}x)^{\gamma}1_{X_{n}\geq t_{n}x}\right]
=\frac{e^{-t_{n}F(x)}}{\sqrt{2\pi t_{n}\eta''(h)}}
\left(\hat{d}_{0}+\frac{\hat{d}_{1}}{t_{n}}+\frac{\hat{d}_{2}}{t_{n}^{2}}
+\cdots+\frac{\hat{d}_{v-1}}{t_{n}^{v-1}}
+O\left(\frac{1}{t_{n}^{v}}\right)\right),
\end{equation*}
we conclude that for every $k\leq v-1$,
\begin{align*}
\hat{d}_{k}
&=\frac{1}{h^{\gamma+1}}
\cdot\sum_{p=0}^{2k}
\sum_{\ell=0}^{p}\frac{\psi^{(p-\ell)}(h)}{(p-\ell)!}
\\
&\qquad
\cdot\sum_{\mathcal{S}_{\ell}}\frac{1}{m_{1}!1!^{m_{1}}m_{2}!2!^{m_{2}}\cdots m_{\ell}!\ell!^{m_{\ell}}}
\cdot\prod_{j=1}^{\ell}\left(\frac{1}{\eta''(h)}\frac{\eta^{(j+2)}(h)}{(j+2)(j+1)}\right)^{m_{j}}
\frac{1}{(\eta''(h))^{p/2}}
\\
&\qquad\cdot
\sum_{m=0}^{\lfloor\frac{p}{2}+m_{1}+\cdots+m_{\ell}\rfloor}
\frac{2^{m}(-1)^{m}(p+2(m_{1}+\cdots+m_{\ell}))!}{m!(p+2(m_{1}+\cdots+m_{\ell})-2m)!}
\\
&\qquad\qquad\qquad\cdot
\frac{\Gamma(k+\gamma+(m_{1}+\cdots+m_{\ell})-m+1)a_{k+m-p-(m_{1}+\cdots+m_{\ell})}}{(h^{2}\eta''(h))^{-\frac{p}{2}+k}}.
\end{align*}
So we have obtained the formula of $\hat d_k$ when $g(z)= z^{\gamma}$.
Finally, assume $g(z)$
has the expansion $g(z)=\sum_{q=0}^{\infty}g_{q}z^{q+\Delta}$, then the formula of $\hat d_k$ in Part (ii) of Proposition~\ref{prop:modphi-nonlattice-gz} readily follows.


To complete the proof, it only remains to show that
\begin{equation*}
\sup_{x\in\mathbb{R}}|F_{n}(x)-G_{n}(x)|=o(t_{n}^{-v}).
\end{equation*}
where $G_n(x) = \int_{-\infty}^x g_n(y) dy$ with $g_n(\cdot)$ given in \eqref{fny}.
We apply the Esseen's smoothing inequality in Appendix~\ref{sec:smooth} and use a similar argument as the proof of Proposition~4.1 in \cite{modphi}.

%
%
%
Recall from \eqref{eq:fn-gn} that
\begin{equation*}
f_{n}^{\ast}(\zeta)=\mathbf{g}_{n}^{\ast}(\zeta)
+e^{\frac{z^{2}}{2}}P_{2v+1}(|z|)\left(\frac{|z|}{\sqrt{t_{n}}}\right)^{2v}\varepsilon\left(\frac{z}{\sqrt{t_{n}}}\right),
\end{equation*}
where $f_{n}^{\ast}$ and $\mathbf{g}_{n}^{\ast}$ are the corresponding Fourier transforms of $F_n$ and $G_n$,
$\lim_{t\rightarrow 0}\varepsilon(t)=0$, and $P_{i}(|z|)$ is a polynomial of $|z|$ of order $i$. Write $k=2v$.
{We can check that $F_{n}-G_{n}$ vanish at $\pm\infty$,
since $F_{n}$ and $G_n$ both vanish at $-\infty$ and
$\lim_{x \rightarrow +\infty} F_{n}(x) = \lim_{x \rightarrow +\infty} G_{n}(x) =1$ as $F_n$ is a cumulative distribution function and $\mathbf{g}_{n}^{\ast}(0)=1$ from \eqref{eq:approx-F}.
In addition, we have
\begin{equation*}
\sup_{n\in\mathbb{N}}\sup_{x\in\mathbb{R}}|G'_{n}(x)|=\sup_{n\in\mathbb{N}}\sup_{x\in\mathbb{R}}g_{n}(x)
=:m<\infty,
\end{equation*}
from the definition of $g_{n}$ in \eqref{fny}
and the fact that $e^{-\frac{y^{2}}{2}}H_{k+2(m_{1}+\cdots+m_{\ell})}(y)\frac{1}{t_{n}^{k/2}}$
is uniformly bounded in $n\in\mathbb{N}$, $y\in\mathbb{R}$ and $m_{1}\cdot 1+\cdots+m_{\ell}\ell=\ell$, where
$\ell\leq k\leq 2v$, where $v\in\mathbb{N}$ is fixed.}
By taking $T=Mt_{n}^{k/2}$ and fixing $\delta \in (0, M)$, we deduce from Esseen's smoothing inequality that
\begin{align}
|F_{n}(x)-G_{n}(x)|
&\leq\frac{1}{\pi}\int_{-Mt_{n}^{k/2}}^{Mt_{n}^{k/2}}
\left|\frac{f^{\ast}_{n}(\zeta)-\mathbf{g}^{\ast}_{n}(\zeta)}{\zeta}\right|d\zeta
+r(\pi) \frac{m}{ Mt_{n}^{k/2}}
\nonumber \\
&\leq
\frac{1}{\pi}\int_{-\delta t_{n}^{1/2}}^{\delta t_{n}^{1/2}}
\left|\frac{f^{\ast}_{n}(\zeta)-\mathbf{g}^{\ast}_{n}(\zeta)}{\zeta}\right|d\zeta
+r(\pi)  \frac{m}{ Mt_{n}^{k/2}}
\nonumber \\
&\qquad
+\frac{1}{\pi}\int_{\left[-M t_{n}^{k/2},Mt_{n}^{k/2}\right]\big\backslash\left[-\delta t_{n}^{1/2},\delta t_{n}^{1/2}\right]}\left|\frac{f^{\ast}_{n}(\zeta)-\mathbf{g}^{\ast}_{n}(\zeta)}{\zeta}\right|d\zeta. \label{eq:essen}
\end{align}
We next estimate each of the three terms in the right hand side of the above inequality.
We can estimate that
\begin{align*}
\frac{1}{\pi}\int_{-\delta t_{n}^{1/2}}^{\delta t_{n}^{1/2}}
\left|\frac{f^{\ast}_{n}(\zeta)-\mathbf{g}^{\ast}_{n}(\zeta)}{\zeta}\right|d\zeta
&=\frac{1}{t_{n}^{k/2}}\frac{1}{\pi}\int_{-\delta t_{n}^{1/2}}^{\delta t_{n}^{1/2}}
e^{-\zeta^{2}/2}P_{k+1}(|\zeta|)|\zeta|^{k-1}\varepsilon\left(\frac{\zeta}{\sqrt{t_{n}}}\right)d\zeta
\\
&\leq\frac{C_{0}}{t_{n}^{k/2}}\max_{-\delta\leq t\leq\delta}\varepsilon(t),
\end{align*}
{for some constant $C_0$ (that depends on $k$). For fixed small $\epsilon>0$, one can choose $\delta$ so that $\max_{-\delta\leq t\leq\delta}\varepsilon(t) < \epsilon$ and $m \delta < \epsilon$, and hence we have for large $t_n$
\begin{equation*}
\frac{1}{\pi}\int_{-\delta t_{n}^{1/2}}^{\delta t_{n}^{1/2}}
\left|\frac{f^{\ast}_{n}(\zeta)-\mathbf{g}^{\ast}_{n}(\zeta)}{\zeta}\right|d\zeta
=o\left(\frac{1}{t_{n}^{k/2}}\right),
\end{equation*}
uniformly in $x$. In addition, by taking $M= 1/\delta$ we get}
\begin{equation*}
r(\pi)  \frac{m}{ Mt_{n}^{k/2}}=o\left(\frac{1}{t_{n}^{k/2}}\right).
\end{equation*}

We can also estimate that
\begin{align}
&\frac{1}{\pi}\int_{[-M t_{n}^{k/2},Mt_{n}^{k/2}]\backslash[-\delta t_{n}^{1/2},\delta t_{n}^{1/2}]}\left|\frac{f^{\ast}_{n}(\zeta)-\mathbf{g}^{\ast}_{n}(\zeta)}{\zeta}\right|d\zeta
\nonumber \\
&\leq\frac{1}{\pi\delta t_{n}^{1/2}}
\left[\int_{\delta t_{n}^{1/2}\leq|\zeta|\leq Mt_{n}^{k/2}}|f^{\ast}_{n}(\zeta)|d\zeta
+\int_{\delta t_{n}^{1/2}\leq|\zeta|\leq Mt_{n}^{k/2}}|\mathbf{g}^{\ast}_{n}(\zeta)|d\zeta\right], \label{eq:3rdterm}
\end{align}
and we can compute that
\begin{align*}
&\int_{\delta t_{n}^{1/2}\leq|\zeta|\leq Mt_{n}^{k/2}}|f^{\ast}_{n}(\zeta)|d\zeta
\\
&\leq\int_{\delta t_{n}^{1/2}\leq|\zeta|\leq Mt_{n}^{k/2}}
\left|\psi_{n}\left(\frac{i\zeta}{\sqrt{t_{n}\eta''(0)}}\right)\right|
\left|e^{t_{n}\eta(\frac{i\zeta}{\sqrt{t_{n}\eta''(0)}})}\right|d\zeta
\\
&\leq \sqrt{t_n} \int_{\delta\leq|\zeta|\leq Mt_{n}^{(k-1)/2}}\left|\psi_{n}\left(\frac{i\zeta}{\sqrt{\eta''(0)}}\right)\right|d\zeta \cdot
\left(\max_{|\zeta|\geq\delta}  \left|\exp\left\{\eta\left(\frac{i\zeta}{\sqrt{\eta''(0)}}\right)\right\}\right|\right)^{t_{n}}.
\end{align*}
Since $\phi$ is non-lattice with $\int_{\mathbb{R}} e^{zx}\phi(dx) = e^{\eta(z)}$, by Lemma 4.9. in \cite{modphi}, we obtain
\begin{equation*}
\max_{|\zeta|\geq\delta}\left|\exp\left\{\eta\left(\frac{i\zeta}{\sqrt{\eta''(0)}}\right)\right\}\right|<1.
\end{equation*}
Then it follows that
\begin{align*}
&\int_{\delta t_{n}^{1/2}\leq|\zeta|\leq Mt_{n}^{k/2}}|f^{\ast}_{n}(\zeta)|d\zeta
\\
&\leq\limsup_{n\rightarrow\infty}\sup_{x\in\mathbb{R}}\left|\psi_{n}\left(ix\right)\right|
\cdot 2Mt_{n}^{k/2}
\cdot\left(\max_{|\zeta|\geq\delta}\left|\exp\left\{\eta\left(\frac{i\zeta}{\sqrt{\eta''(0)}}\right)\right\}\right|\right)^{t_{n}},
\end{align*}
which goes to zero faster than any power of $t_{n}$.
Finally, the term $\int_{\delta t_{n}^{1/2}\leq|\zeta|\leq Mt_{n}^{k/2}}|\mathbf{g}^{\ast}_{n}(\zeta)|d\zeta$ in \eqref{eq:3rdterm} can be estimated similarly.
Therefore, we deduce from \eqref{eq:essen} that for large $t_n$,
\begin{equation*}
\sup_{x\in\mathbb{R}}|F_{n}(x)-G_{n}(x)| = o\left(\frac{1}{t_{n}^{k/2}}\right) =
o\left(\frac{1}{t_{n}^{v}}\right).
\end{equation*}
The proof is hence complete.

\subsubsection{Proof of Lemma~\ref{yintegral2}}

\begin{proof}[Proof of Lemma \ref{yintegral2}]
For any $s\geq 0$, it is readily seen that
\begin{equation*}
\int_{y=0}^{\infty}y^{s}e^{-\frac{(y+h\sqrt{t_{n}\eta''(h)})^{2}}{2}}dy
=\left(h\sqrt{t_{n}\eta''(h)}\right)^{s+1}\int_{y=0}^{\infty}y^{s}e^{-h^{2}t_{n}\eta''(h)\frac{(y+1)^{2}}{2}}dy.
\end{equation*}
For large $t_n$, we can apply the Laplace's method (see Lemma~\ref{Laplace'sThm} in Appendix~\ref{sec:Laplace's}) to estimate the above integral. To this end, we use the notation in Lemma~\ref{Laplace'sThm}, and we set
\begin{equation*}
p(y)=\frac{1}{2}(y+1)^{2},
\qquad
q(y)=y^{s},
\end{equation*}
and $a=0$, $b=\infty$ and $\lambda=h^{2}t_{n}\eta''(h)$.
We can compute that
\begin{equation*}
p(y)=\frac{1}{2}+y+\frac{1}{2}y^{2},
\end{equation*}
which implies that $\mu=1$.
Since $q(y)=y^{s}$, we have $\sigma=s+1$.
Moreover,
\begin{equation*}
v=p(y)-p(0)=y+\frac{1}{2}y^{2},
\end{equation*}
and $(a_{k})_{k=0}^{\infty}$ are determined from:
\begin{equation*}
\frac{q(y)}{p'(y)}
=\frac{y^{s}}{1+y}
\sim\sum_{k=0}^{\infty}a_{k}\left(y+\frac{1}{2}y^{2}\right)^{k+s},\qquad y\rightarrow 0^{+}.
\end{equation*}
Hence, as $t_{n}\rightarrow\infty$,
\begin{equation*}
\int_{y=0}^{\infty}y^{s}e^{-h^{2}t_{n}\eta''(h)\frac{(y+1)^{2}}{2}}dy
\sim e^{-h^{2}t_{n}\eta''(h)\frac{1}{2}}
\sum_{k=0}^{\infty}\frac{\Gamma(k+s+1)a_{k}}{(h^{2}t_{n}\eta''(h))^{s+1+k}}.
\end{equation*}
Finally, we can compute that
\begin{equation*}
\frac{y^{s}}{1+y}
\sim\sum_{k=0}^{\infty}a_{k}\left(y+\frac{1}{2}y^{2}\right)^{k+s}
=\sum_{k=0}^{\infty}a_{k}y^{s+k}\left(1+\frac{1}{2}y\right)^{k+s},
\end{equation*}
which implies that for $y\rightarrow 0^{+}$,
\begin{equation*}
\frac{1}{1+y}
=\sum_{k=0}^{\infty}(-1)^{k}y^{k}
\sim\sum_{j=0}^{\infty}a_{j}y^{j}\left(1+\frac{1}{2}y\right)^{j+s}
=\sum_{j=0}^{\infty}a_{j}y^{j}\sum_{i=0}^{\infty}\frac{1}{2^{i}}y^{i}\frac{(j+s)_{i}}{i!},
\end{equation*}
where $(\cdot)_{i}$ is the Pochhammer symbol. This implies that
\begin{equation*}
(-1)^{k}=\sum_{j=0}^{k}\frac{(j+s)_{k-j}}{(k-j)!2^{k-j}}a_{j}
=a_{k}+\sum_{j=0}^{k-1}\frac{(j+s)_{k-j}}{(k-j)!2^{k-j}}a_{j}.
\end{equation*}
Hence, the conclusion follows.
\end{proof}


\section{Computations of expansion coefficients for affine point processes}\label{sec:affine-coeff}

In this section, we compute $\hat c_{k}$ and
$\hat d_k$ for affine point processes. Their general formulas are given in Propositions~\ref{prop:modphi-lattice} and
\ref{prop:modphi-nonlattice-gz}. It is clear that these coefficients depend on the derivatives
of $\eta$ and $\psi$ at $h$. Recall that for affine point processes, the function $\psi(\cdot)$ is given in \eqref{psiDefn}, $\eta(\cdot)$ is given in \eqref{eq:eta-phi},
and the real number $h$ is defined by $\eta'(h)=R$.
We next discuss how to compute high order derivatives for functions $\psi$ and $\eta$ that are associated with affine point processes.

\subsection{Computations of $\psi^{(k)}(\theta)$}

From \eqref{psiDefn}, we have $\psi(\theta)=e^{u^{\ast}(\theta)^{T}x_{0}+B(\infty;\theta,u^{\ast}(\theta))}.$
By applying Fa\`{a} di Bruno's formula (see Section \ref{sec:Faa}), we get the $k-$th derivative of $\psi$ is given by
\begin{equation*}
\psi^{(k)}(\theta)
=\sum_{\mathcal{S}_{k}}\frac{k!\psi(\theta)}{m_{1}!1!^{m_{1}}m_{2}!2!^{m_{2}}\cdots m_{k}!k!^{m_{k}}}
\cdot
\prod_{j=1}^{k}\left(\left((u^{\ast})^{(j)} (\theta) \right)^{T}x_{0}
+\frac{d^{j}}{d\theta^{j}}B(\infty;\theta,u^{\ast}(\theta))\right)^{m_{j}}.
\end{equation*}
Hence, it suffices for us to compute the derivatives  $\frac{d^{j}}{d\theta^{j}}B(\infty;\theta,u^{\ast}(\theta))$
which will be presented in Section \ref{BkSection}, and the derivatives of $u^{\ast}(\theta)$ with respect to $\theta$ (up to order $k$), which will be presented in Section \ref{UkSection}.

\subsection{Computations of $\eta^{(k)}(\theta)$}

We recall from \eqref{eq:eta-phi} that
\begin{equation*}
\eta(\theta)=u^{\ast}(\theta)^{T}b
+\sum_{i=1}^{n}\lambda_{i}\left(\mathbb{E}\left[e^{\left(\theta+u^{\ast}(\theta)^{T}\gamma_{i}\right)Z_i}\right]-1\right).
\end{equation*}
For any $k\geq 1$, by Leibnitz formula, we can compute the $k-$th derivative of $\eta$:
\begin{align*}
\eta^{(k)}(\theta)
&=(u^{\ast})^{(k)}(\theta)^{T}b
+\sum_{i=1}^{n}\lambda_{i}\mathbb{E}
\left[\frac{d^{k}}{d\theta^{k}}e^{\theta Z_{i}}e^{u^{\ast}(\theta)^{T}\gamma_{i}Z_i}\right]
\\
&=(u^{\ast})^{(k)}(\theta)^{T}b
+\sum_{i=1}^{n}\lambda_{i}\mathbb{E}
\left[\sum_{j=0}^{k}\binom{k}{j}Z_{i}^{k-j}e^{\theta Z_{i}}\frac{d^{j}}{d\theta^{j}}e^{u^{\ast}(\theta)^{T}\gamma_{i}Z_i}\right],
\end{align*}
and by Fa\`{a} di Bruno's formula,
\begin{align*}
\frac{d^{j}}{d\theta^{j}}e^{u^{\ast}(\theta)^{T}\gamma_{i}Z_i}
&=\sum_{\mathcal{S}_{j}}\frac{j!e^{u^{\ast}(\theta)^{T}\gamma_{i}Z_i}}
{m_{1}!1!^{m_{1}}m_{2}!2!^{m_{2}}\cdots m_{j}!j!^{m_{j}}}
\cdot
\prod_{\ell=1}^{n}\left((u^{\ast})^{(\ell)}(\theta)^{T}\gamma_{i}Z_i\right)^{m_{\ell}},
\end{align*}
which yields that
\begin{align*}
\eta^{(k)}(\theta)
&=(u^{\ast})^{(k)}(\theta)^{T}b
+\sum_{i=1}^{n}\lambda_{i}\mathbb{E}
\left[\frac{d^{k}}{d\theta^{k}}e^{\theta Z_{i}}e^{u^{\ast}(\theta)^{T}\gamma_{i}Z_i}\right]
\\
&=(u^{\ast})^{(k)}(\theta)^{T}b
\\
&\qquad\qquad
+\sum_{i=1}^{n}\lambda_{i}
\sum_{j=0}^{k}\binom{k}{j}\sum_{\mathcal{S}_{j}}\frac{j!
\mathbb{E}[Z_{i}^{k-j}e^{\theta Z_{i}}e^{u^{\ast}(\theta)^{T}\gamma_{i}Z_i}
\prod_{\ell=1}^{n}((u^{\ast})^{(\ell)}(\theta)^{T}\gamma_{i}Z_i)^{m_{\ell}}]}
{m_{1}!1!^{m_{1}}m_{2}!2!^{m_{2}}\cdots m_{j}!j!^{m_{j}}}.
\end{align*}
Hence, to find the $k-$th derivative of $\eta$, we need to compute the derivatives of $u^{\ast}(\theta)$ with respect to $\theta$ (up to order $k$),
which we present in Section \ref{UkSection}.

\subsection{Computations of $\frac{d^{j}}{d\theta^{j}}B(\infty;\theta,u^{\ast}(\theta))$}\label{BkSection}

By applying multivariate Fa\`{a} di Bruno's formula and using the notations in Appendix \ref{sec:Faa}, we have
\begin{equation} \label{eq:B-deriv-1}
\frac{d^{j}}{d\theta^{j}}B(\infty;\theta,u^{\ast}(\theta))
=\sum_{1\leq|\mathbf{\lambda}|\leq j}
D_{\hat{\delta}}^{\mathbf{\lambda}}B(\infty;\hat{\delta})
\sum_{p(j,\mathbf{\lambda})}j!
\prod_{m=1}^{j}\frac{[(\hat{u}^{\ast})^{(\ell_{m})}(\theta)]^{\mathbf{k}_{m}}}
{(\mathbf{k}_{m}!)[\ell_{m}!]^{|\mathbf{k}_{m}|}},
\end{equation}
where $\hat{\delta}=\hat{u}^{\ast}(\theta):=(u_0^{\ast}(\theta), u_1^{\ast}(\theta), u_2^{\ast}(\theta), ..., u_d^{\ast}(\theta))^T$, ${u}^{\ast}(\theta)=(u_1^{\ast}(\theta), u_2^{\ast}(\theta), ..., u_d^{\ast}(\theta))^T$ and 
$u_0^{\ast}(\theta) := \theta$.

We next compute $D_{\hat{\delta}}^{\mathbf{\nu}}B(\infty;\hat{\delta})$ for a multi-index $\nu \ne \mathbf{0}$. 
Recall from Theorem~\ref{thm:app-mod} that,
\begin{align*}
B(\infty;\hat{\delta})&=\int_{0}^{\infty}b^{T}A dt+\int_{0}^{\infty}\frac{1}{2}A^{T}aAdt
\nonumber \\
&\qquad\qquad
+\sum_{i=1}^{n}\lambda_{i}\int_{0}^{\infty}\int_{\mathbb{R}_{+}}\left(e^{A^{T}\gamma_{i}z}-1\right)e^{(\hat{\delta}_0+\delta^{T}\gamma_{i})z}\varphi_{i}(dz)dt,
\end{align*}
where for notational simplicity, we write
$A : =A(t;\hat{\delta})=A(t;\theta,\delta)=(A_{1}(t;\theta,\delta),\cdots,A_{n}(t;\theta,\delta))$,
with $\hat{\delta}=(\hat{\delta}_{0},\delta)$ 
so that $\hat{\delta}_{0}=\theta$, $\delta= u^{\ast}(\theta)$, $A(0;\hat{\delta})=A(0;\theta,\delta)=-\delta$ and
\begin{equation} \label{eq:ODE-A}
\frac{d}{dt}A_{j}(t;\hat{\delta})
=-\sum_{i=1}^{n}A_{i}\beta_{i,j}^{\ast}+\frac{1}{2}A^{T}\alpha^{j}A+\sum_{i=1}^{n}
\int_{\mathbb{R}_{+}}\left(e^{A^{T}\gamma_{i}z}-1\right)e^{(\hat{\delta}_0+\delta^{T}\gamma_{i})z}\varphi_{i}(dz)\kappa_{i,j},
\end{equation}
for $j=1,2,\ldots,n$, where
\begin{equation*}
\beta^{\ast}=
\left(
\begin{array}{cc}
\beta_{I,I}-\text{diag}(\alpha_{11}^{1}u^{\ast}_{1}(\theta),\ldots,\alpha_{mm}^{m}u_{m}^{\ast}(\theta)) & 0
\\
\beta_{J,I} & \beta_{J,J}
\end{array}
\right).
\end{equation*}
By Applying Leibnitz formula and multivariate Fa\`{a} di Bruno's formula (see Appendix \ref{sec:Faa}), we get
\begin{align} \label{eq:B-deriv-2}
&D_{\hat{\delta}}^{\mathbf{\nu}}B(\infty; \hat{\delta})
\nonumber
\\
&=\int_{0}^{\infty}b^{T}D_{\hat{\delta}}^{\mathbf{\nu}}A dt
+\int_{0}^{\infty}\frac{1}{2}
\sum_{\mathbf{\lambda}\leq\mathbf{\nu}}\binom{\mathbf{\nu}}{\mathbf{\lambda}}
(D_{\hat{\delta}}^{\mathbf{\lambda}}A)^{T}a(D_{\hat{\delta}}^{\mathbf{\nu}-\mathbf{\lambda}}A)dt
\nonumber \\
&\qquad
+\sum_{i=1}^{n}\lambda_{i}\int_{0}^{\infty}\int_{\mathbb{R}_{+}}
\Bigg[\sum_{0<\tilde{\nu}\leq\nu}\binom{\nu}{\tilde{\nu}}
\sum_{1\leq\lambda\leq|\mathbf{\tilde{\nu}}|}e^{A^{T}\gamma_{i}z}
\sum_{p(\mathbf{\tilde{\nu}},\lambda)}\mathbf{\tilde{\nu}}!
\prod_{j=1}^{|\mathbf{\tilde{\nu}}|}\frac{[D_{\hat{\delta}}^{\Bell_{j}}A^{T}\gamma_{i}z]^{k_{j}}}
{(k_{j}!)[\Bell_{j}!]^{k_{j}}}
D_{\hat{\delta}}^{\nu-\tilde{\nu}}e^{(\hat{\delta}_0+\delta^{T}\gamma_{i})z}
\nonumber
\\
&\qquad\qquad\qquad\qquad\qquad
+\left(e^{A^{T}\gamma_{i}z}-1\right)D_{\hat{\delta}}^{\nu}e^{(\hat{\delta}_0+\delta^{T}\gamma_{i})z}\Bigg]\varphi_{i}(dz)dt,
\end{align}
where for two multi-indices $\lambda, \nu$, we write $\lambda \le \nu$ if $\lambda_i \le \nu_i$ for all $i$.

So to compute $D_{\hat{\delta}}^{\mathbf{\nu}}B(\infty;\hat{\delta})$, it remains to compute
$D_{\hat{\delta}}^{\mathbf{\nu}}A_{j}(t;\hat{\delta})$ for each $j$ for a multi-index $\nu$. From the ODE for $A_j(t;\hat{\delta})$ in \eqref{eq:ODE-A}, one can readily obtain the following ODE for $D_{\hat{\delta}}^{\mathbf{\nu}}A_{j}(t;\hat{\delta})$:
\begin{align} \label{eq:B-deriv-3}
&\frac{d}{dt}D_{\hat{\delta}}^{\mathbf{\nu}}A_{j}(t;\hat{\delta})
\nonumber
\\
&=-\sum_{i=1}^{n}D_{\hat{\delta}}^{\mathbf{\nu}}A_{i}\beta_{i,j}^{\ast}
+\int_{0}^{\infty}\frac{1}{2}
\sum_{\mathbf{\lambda}\leq\mathbf{\nu}}\binom{\mathbf{\nu}}{\mathbf{\lambda}}
(D_{\hat{\delta}}^{\mathbf{\lambda}}A)^{T}\alpha^{j}(D_{\hat{\delta}}^{\mathbf{\nu}-\mathbf{\lambda}}A)
\nonumber
\\
&\qquad
+\sum_{i=1}^{n}\int_{\mathbb{R}_{+}}
\sum_{1\leq\lambda\leq|\mathbf{\nu}|}e^{A^{T}\gamma_{i}z}
\Bigg[\sum_{0<\tilde{\nu}\leq\nu}\binom{\nu}{\tilde{\nu}}
\sum_{p(\mathbf{\tilde{\nu}},\lambda)}\mathbf{\tilde{\nu}}!
\prod_{j=1}^{|\mathbf{\tilde{\nu}}|}\frac{[D_{\hat{\delta}}^{\Bell_{j}}A^{T}\gamma_{i}z]^{k_{j}}}
{(k_{j}!)[\Bell_{j}!]^{k_{j}}}
D_{\hat{\delta}}^{\nu-\tilde{\nu}}e^{(\hat{\delta}_0+\delta^{T}\gamma_{i})z}
\nonumber
\\
&\qquad\qquad\qquad\qquad\qquad
+\left(e^{A^{T}\gamma_{i}z}-1\right)D_{\hat{\delta}}^{\nu}e^{(\hat{\delta}_0+\delta^{T}\gamma_{i})z}\Bigg]\varphi_{i}(dz)\kappa_{i,j},
\end{align}
for $j=1,2,\ldots,n$, with the initial condition $D_{\hat{\delta}}^{\mathbf{\nu}}A(0;\hat{\delta})=-D_{\hat{\delta}}^{\mathbf{\nu}}\delta$ as $A(0;\hat{\delta})=-\delta$.

In summary, one can numerically solve the ODEs in \eqref{eq:B-deriv-3} to obtain $D_{\hat{\delta}}^{\mathbf{\nu}}A_{j}(t;\hat{\delta})$ for each $j$, next numerically compute $D_{\hat{\delta}}^{\mathbf{\nu}}B(\infty;\hat{\delta})$ using \eqref{eq:B-deriv-2}, and finally obtain $\frac{d^{j}}{d\theta^{j}}B(\infty;\theta,u^{\ast}(\theta))$ using \eqref{eq:B-deriv-1}.

\subsection{Computations of $(u^{\ast})^{(k)}(\theta)$}\label{UkSection}

We recall from Theorem~\ref{thm:zhang2} that
$u^{\ast}(\cdot):\mathbb{R}\rightarrow\mathbb{R}^{n}$ is the implicit function defined
as the unique solution branch with $u^{\ast}(0)=0$ of the system of nonlinear equations:
\begin{equation} \label{eq:u-0}
\beta^{T}u^{\ast}(\theta)-\frac{1}{2}u^{\ast}(\theta)^{T}\alpha u^{\ast}(\theta)-
\kappa^{T}\left(\mathbb{E}\left[e^{(\theta+(u^{\ast}(\theta) )^{T}\gamma)Z}\right]-1\right)=0.
\end{equation}
To obtain the $k-$th derivative of $u^{\ast}(\theta)$ with respect to $\theta$, that is, $(u^{\ast})^{(k)}(\theta)$, we can
apply Leibnitz formula and Fa\`{a} di Bruno's formula (see Appendix \ref{sec:Faa}) to the above equation and get
\begin{align}
&\beta^{T}(u^{\ast})^{(k)}(\theta)
-\frac{1}{2}\sum_{j=0}^{k}\binom{k}{j}((u^{\ast})^{(k-j)} (\theta) )^{T}\alpha(u^{\ast})^{(j)}(\theta)
\nonumber \\
&\qquad
=\kappa^{T}
\sum_{j=0}^{k}\binom{k}{j}\sum_{\mathcal{S}_{j}}\frac{j!
\mathbb{E}[Z^{k-j}e^{\theta Z}e^{u^{\ast}(\theta)^{T}\gamma Z}
\prod_{\ell=1}^{j}((u^{\ast})^{(\ell)}(\theta)^{T}\gamma Z)^{m_{\ell}}]}
{m_{1}!1!^{m_{1}}m_{2}!2!^{m_{2}}\cdots m_{j}!j!^{m_{j}}}. \label{eq:u-k}
\end{align}
Therefore, given the parameters, one can solve the nonlinear equation in \eqref{eq:u-0} to obtain $u^{\ast}(\theta)$, and then successively solve the nonlinear equations in \eqref{eq:u-k} to obtain the derivatives $(u^{\ast})^{(k)}(\theta)$ for any $k \ge 1$.

\section{Smoothing inequality, Fa\`{a} di Bruno's formula and Laplace's Method} \label{sec:known}
For completeness, this section collects three known results that are used in our proofs.

\subsection{Esseen's smoothing inequality} \label{sec:smooth}
\begin{lemma} [Esseen's smoothing inequality, cf. Theorem~2 in Chapter V of \cite{Petrov}]

Let $F$
be a non-decreasing function and $G$ be a differentiable function of bounded variation on the real line with respective Fourier-Stieltjes transforms $f^{\ast}$ and $g^{\ast}$. Suppose that $F-G$ vanishes at $\pm \infty$ and that $G$ is $m$-Lipschitz with $\sup_x|G'(x)| \le m$. Then for every $T>0$,

\begin{equation*}
\sup_{x \in \mathbb{R}} |F(x)-G(x)|\leq\frac{1}{\pi}\int_{-T}^{T}\left|\frac{f^{\ast}(\zeta)-{g}^{\ast}(\zeta)}{\zeta}\right|d\zeta
+r(\pi) \frac{m}{T},
\end{equation*}
where $r(\pi)$ is a positive constant depending only on $\pi$.
\end{lemma}
\subsection{Fa\`{a} di Bruno's formula}\label{sec:Faa}


(Multivariate) Faà di Bruno's formula gives an explicit equation for the higher order (partial) derivatives of the composition
$h(x_{1},\ldots,x_{d})=f(g_{1}(x_{1},\ldots,x_{d}),\ldots,g_{m}(x_{1},\ldots,x_{d})),$
where $g_i:\mathbb{R}^{d}\rightarrow\mathbb{R}$
and $f:\mathbb{R}^{m}\rightarrow\mathbb{R}$ are differentiable a sufficient number of times.

To introduce the formula, we recall some multivariate notations.
For $\mathbf{\nu}\in(\mathbb{N}\cup\{0\})^{d}$
and $\mathbf{x}\in\mathbb{R}^{d}$, we define
$|\mathbf{\nu}|=\sum_{i=1}^{d}\nu_{i}$, $\mathbf{\nu}!=\prod_{i=1}^{d}\nu_{i}!$,
$D^{\mathbf{\nu}}_{\mathbf{x}}=\frac{\partial^{|\mathbf{\nu}|}}{\partial x_{1}^{\nu_{1}}\cdots\partial x_{d}^{\nu_{d}}}$
for $|\mathbf{\nu}|>0$, and $\mathbf{x}^{\nu}=\prod_{i=1}^{d}x_{i}^{\nu_{i}}$. In addition,
if $\mathbf{\mu}=(\mu_{1},\ldots,\mu_{d})$
and $\mathbf{\nu}=(\nu_{1},\ldots,\nu_{d})$ are both in $(\mathbb{N}\cup\{0\})^{d}$,
we write $\mathbf{\mu}\prec\mathbf{\nu}$ provided one of the following holds:
(i) $|\mathbf{\mu}|<\mathbf{\nu}|$;
(ii) $|\mathbf{\mu}|=|\mathbf{\nu}|$ and $\mu_{1}<\nu_{1}$ or
(iii) $|\mathbf{\mu}|=|\mathbf{\nu}|$, $\mu_{1}=\nu_{1},\ldots,\mu_{k}=\nu_{k}$
and $\mu_{k+1}<\nu_{k+1}$ for some $1\leq k<d$.

Let $\mathbf{\nu}=(\nu_{1},\ldots,\nu_{d})\neq(0,\ldots,0)$,
Setting $h_{\mathbf{\nu}}=D_{\mathbf{x}}^{\mathbf{\nu}}h(\mathbf{x})$,
$f_{\mathbf{\lambda}}=D_{\mathbf{y}}^{\mathbf{\lambda}}f(\mathbf{y})$,
where $\mathbf{y}=(g_{1}(\mathbf{x}),\ldots,g_{m}(\mathbf{x}))$
and $\mathbf{\nu},\mathbf{\lambda}\in(\mathbb{N}\cup\{0\})^{d}$,
and $\mathbf{g}_{\mathbf{\mu}}=(D_{\mathbf{x}}^{\mathbf{\mu}}g_{1}(\mathbf{x}),\ldots,D_{\mathbf{x}}^{\mathbf{\mu}}g_{m}(\mathbf{x}))$. Then, the multivariate Fa\`{a} di Bruno's formula is given as follows.

\begin{lemma}[Multivariate Fa\`{a} di Bruno's formula, cf.  (2.4) in \cite{CS96}]
\begin{equation*}
h_{\mathbf{\nu}}=\sum_{1\leq|\mathbf{\lambda}|\leq|\mathbf{\nu}|}f_{\mathbf{\lambda}}
\sum_{p(\mathbf{\nu},\mathbf{\lambda})}\mathbf{\nu}!
\prod_{j=1}^{|\mathbf{\nu}|}\frac{[\mathbf{g}_{\Bell_{j}}]^{\mathbf{k}_{j}}}
{(\mathbf{k}_{j}!)[\Bell_{j}!]^{|\mathbf{k}_{j}|}},
\end{equation*}
where
\begin{align*}
p(\mathbf{\nu},\mathbf{\lambda})
&=\bigg\{\left(\mathbf{k}_{1},\ldots,\mathbf{k}_{|\mathbf{\nu}|};\Bell_{1},\ldots,\Bell_{|\mathbf{\nu}|}\right):
\text{for some $1\leq s\leq |\mathbf{\nu}|$,}
\\
&\qquad
\text{$\mathbf{k}_{i}=0$ and $\Bell_{i}=\mathbf{0}$ for $1\leq i\leq |\mathbf{\nu}|-s$;
$|\mathbf{k}_{i}|>0$ for $|\mathbf{\nu}|-s+1\leq i\leq |\mathbf{\nu}|$;}
\\
&\qquad
\text{and $\mathbf{0}\prec\Bell_{|\mathbf{\nu}|-s+1}\prec\cdots\prec\Bell_{|\mathbf{\nu}|}$ are such that }
\sum_{i=1}^{|\mathbf{\nu}|}\mathbf{k}_{i}=\mathbf{\lambda},
\sum_{i=1}^{|\mathbf{\nu}|}|\mathbf{k}_{i}|{\Bell}_{i}=\mathbf{\nu}\bigg\}.
\end{align*}
\end{lemma}
In the above formula, the vectors $\mathbf{k}$ are $m-$dimensional, the vectors $\Bell$ are $d-$dimensional, and we always set $0^{0}=1$.

For the special case $d=m=1$, we have the one-dimensional Fa\`{a} di Bruno's formula,
which has a more explicit expression:

\begin{lemma}[Univariate Fa\`{a} di Bruno's formula] \label{lem:faadi}
\begin{equation*}
\frac{d^{n}}{dx^{n}}f(g(x))
=\sum_{\mathcal{S}_{n}}\frac{n!}{m_{1}!1!^{m_{1}}m_{2}!2!^{m_{2}}\cdots m_{n}!n!^{m_{n}}}
\cdot
f^{(m_{1}+\cdots+m_{n})}(g(x))\cdot
\prod_{j=1}^{n}\left(g^{(j)}(x)\right)^{m_{j}},
\end{equation*}
where the sum is over the set $\mathcal{S}_{n}$ consisting of
all the $n$-tuples of non-negative integers $(m_{1},\ldots,m_{n})$
satisfying the following constraint:
\begin{equation*}
1\cdot m_{1}+2\cdot m_{2}+3\cdot m_{3}+\cdots+n\cdot m_{n}=n.
\end{equation*}
\end{lemma}


\subsection{Laplace's Method} \label{sec:Laplace's}
Laplace's Method is useful in estimating integrals of the form $I(\lambda)=\int_{a}^{b}e^{-\lambda p(t)}q(t)dt$ as $\lambda \rightarrow \infty$. For two functions $f_1, f_2$, we write $f_1(\lambda) \sim f_2(\lambda )$ if $f_1(\lambda)/f_2(\lambda)$ tends to unity as $\lambda \rightarrow \infty$.
The following result can be found in e.g. \cite[Section 3.7]{Olver}.

 \begin{lemma}[Laplace's Method]\label{Laplace'sThm}
Suppose that
\begin{itemize}
\item [(1)] $p(t)>p(a)$ for $t\in(a,b)$ and the minimum of $p(t)$ is only approached
at $t=a$.

\item [(2)] $p'(t)$, $q'(t)$ are continuous in a neighborhood of $t=a$ except possibly at $t=a$.

\item [(3)] As $t\rightarrow a^{+}$,
\begin{equation*}
p(t)\sim p(a)+\sum_{k=0}^{\infty}p_{k}(t-a)^{k+\mu},
\qquad
q(t)\sim\sum_{k=0}^{\infty}q_{k}(t-a)^{k+\sigma-1},
\end{equation*}
where $\mu,\sigma>0$, $p_{0}\neq 0$ and $q\neq 0$. Also assume
that we can differentiate $p(t)$ and
\begin{equation*}
p'(t)\sim\sum_{k=0}^{\infty}(k+\mu)p_{k}(t-a)^{k+\mu-1}.
\end{equation*}

\item [(4)] $\int_{a}^{b}e^{-\lambda p(t)}q(t)dt$ converges absolutely for all sufficiently large $\lambda$.
\end{itemize}
Then, we have
\begin{equation*}
I(\lambda)=\int_{a}^{b}e^{-\lambda p(t)}q(t)dt\sim e^{-\lambda p(a)}\sum_{k=0}^{\infty}
\Gamma\left(\frac{k+\sigma}{\mu}\right)\frac{a_{k}}{\lambda^{\frac{k+\sigma}{\mu}}},  \qquad
\text{as $\lambda \rightarrow \infty$}.
\end{equation*}
where $v=p(t)-p(a)$ and
\begin{equation*}
f(v)=\frac{q(t)}{p'(t)}\sim\sum_{k=0}^{\infty}a_{k}v^{\frac{k+\sigma-\mu}{\mu}},
\qquad
\text{as $v\rightarrow 0^{+}$}.
\end{equation*}
\end{lemma}


\clearpage
\begin{thebibliography}{31}

\bibitem{AitSahalia2015}
A\"{i}t-Sahalia, Y., Cacho-Diaz, J. and Laeven, R.J., 2015.
Modeling financial contagion using mutually exciting jump processes.
\textit{Journal of Financial Economics}, 117(3), pp.585-606.


\bibitem{Bahadur1960}
Bahadur, R.R. and Rao, R.R., 1960.
On deviations of the sample mean.
\textit{Annals of Mathematical Statistics}, 31(4), pp.1015-1027.

\bibitem{BR2002}
Bercu, B. and A. Rouault, 2002.
Sharp large deviations for the Ornstein-Uhlenbeck process.
\textit{SIAM Theory of Probability and Its Applications},
\textbf{46}, 1-19.

\bibitem{Bordenave2007}
Bordenave, C., and Torrisi, G. L., 2007.
Large deviations of Poisson cluster processes. 
\textit{Stochastic Models}, 23(4), 593-625.

\bibitem{Bowsher2007}
Bowsher, C.G., 2007.
Modelling security market events in continuous time: Intensity based, multivariate point process models.
\textit{Journal of Econometrics}, 141(2), pp.876-912.

\bibitem{CS93}
Chaganty, N. R. and J. Sethuraman, 1993.
Strong large deviation and local limit theorems.
\textit{Annals of Probability},
\textbf{21}, 1671-1690.

\bibitem{CS96}
Constantine, G. M. and T. H. Savits, 1996.
A multivariate Faa di Bruno formula with applications.
\textit{Transactions of the American Mathematical Society},
\textbf{348}. 503-520.

\bibitem{Cramer}
Cram{\'e}r, H., 1938.
Sur un nouveau th{\'e}oreme-limite de la th{\'e}orie des probabilit{\'e}s.
\textit{Actualit{\'e}s scientifiques et industrielles}, 736(5-23), p.115.



\bibitem{Daw2018}
Daw, A. and Pender, J., 2018. 
Queues driven by Hawkes processes. \textit{Stochastic Systems}, 8(3), pp.192-229.

\bibitem{Delbaen2014}
Delbaen, F., Kowalski, E. and Nikeghbali, A., 2014.
Mod-$\phi$ convergence.
\textit{International Mathematics Research Notices}, 2015(11), pp.3445-3485.


\bibitem{Dembo}
Dembo, A. and O. Zeitouni., 1998.
\textit{Large Deviations Techniques and Applications}, 2nd Edition, Springer, New York.

\bibitem{Donsker}
Donsker, M. D. and Varadhan, S. R. S., 1975.
Asymptotic evaluations of certain Markov process expectations for large time, I.
\textit{Communications on Pure and Applied Mathematics},
\textbf{28}, 1-47.



\bibitem{Duffie2000}
Duffie, D., Pan, J. and Singleton, K., 2000.
Transform analysis and asset pricing for affine jump--diffusions.
\textit{Econometrica}, 68(6), pp.1343-1376.


\bibitem{Errais}
Errais, E., K. Giesecke, and Goldberg, L., 2010.
Affine point processes and portfolio credit risk,
\textit{SIAM Journal on Financial Mathematics},
1, 642-665.


\bibitem{modphi}
Feray, V., Meliot, P.-L., and A. Nikeghbali., 2015.
Mod-$\phi$ convergence I: Normality zones and precise deviations.
\textit{arXiv:1304.2934}.



\bibitem{Fox2016}
Fox, E.W., Short, M.B., Schoenberg, F.P., Coronges, K.D. and Bertozzi, A.L., 2016.
Modeling e-mail networks and inferring leadership using self-exciting point processes.
\textit{Journal of the American Statistical Association}, 111(514), pp.564-584.

\bibitem{GaoZhu2016b}
Gao, X. and Zhu, L., 2018.
A functional central limit theorem for stationary Hawkes processes and its application to infinite-server queues.
\textit{Queueing Systems},
\textbf{90}, 161-206.


%

%


\bibitem{Harrington1978}
Harrington, D. P., 1978.
Almost sure convergence of continuous time branching processes.
\textit{Stochastic Processes and their Applications}, 6(3), 317-322.

\bibitem{Hawkes}
Hawkes, A. G., 1971.
Spectra of some self-exciting and mutually exciting point processes,
\textit{Biometrika}, 58, 83-90.

\bibitem{Hewlett2006}
Hewlett. P., 2006.
Clustering of order arrivals, price impact and trade path optimization.
In: Workshop on Financial Modeling with Jump Processes,
Ecole Polytechnique, pp. 6-8.


\bibitem{Hult2010}
Hult, H., and Samorodnitsky, G., 2010.
Large deviations for point processes based on stationary sequences with heavy tails.
\textit{Journal of Applied Probability}, 47(1), 1-40.

\bibitem{Jacod2011}
Jacod, J., Kowalski, E. and Nikeghbali, A., 2011.
Mod-Gaussian convergence: new limit theorems in probability and number theory.
In: \textit{Forum Mathematicum} (Vol. 23, No. 4, pp. 835-873).

\bibitem{Joutard}
Joutard, C., 2013.
Strong large deviations for arbitrary sequences of random variables.
\textit{Annals of the Institute of Statistical Mathematics},
\textbf{65}, 49-67.

\bibitem{Karatzas2012}
Karatzas, I., and Shreve, S., 2012.
\textit{Brownian Motion and Stochastic Calculus}. (Vol. 113). Springer New York.

\bibitem{Koops2018}
Koops, D. T., Mayank Saxena, O. J. Boxma, and Michel Mandjes., 2018. Infinite-server queues with Hawkes input. \textit{Journal of Applied Probability}, 55(3), 920-943.


\bibitem{Olver}
Olver, F., 1997.
\textit{Asymptotics and Special Functions}.
2nd Edition, CRC Press.

\bibitem{Petrov}
Petrov, V.V., 1975. \textit{Sums of Independent Random Variables}.
Springer-Verlag. Berlin.

\bibitem{Pinto2015}
Pinto, J.C.L., Chahed, T. and Altman, E., 2015.
Trend detection in social networks using Hawkes processes.
\textit{Proceedings of the 2015 IEEE/ACM International Conference on Advances in Social Networks Analysis and Mining.} pp. 1441-1448. ACM.



\bibitem{VaradhanLDP}
Varadhan, S. R. S., 1966.
Asymptotic probability and differential equations.
\textit{Communications on Pure and Applied Mathematics},
\textbf{19}, 261-286.

%


\bibitem{Zhang}
Zhang, X., Blanchet, J., Giesecke, K., and P. W. Glynn., 2015.
Affine point processes: Approximation and efficient simulation.
\textit{Mathematics of Operations Research},
\textbf{40}(4), 797-819.


\bibitem{ZhuI}
Zhu, L., 2015.
Large deviations for Markovian nonlinear Hawkes Processes.
\textit{Annals of Applied Probability},
\textbf{25}, 548-581.
\end{thebibliography}
\end{document}